\documentclass[reqno]{amsart} 
\usepackage{amsfonts, amsmath, amsthm, amssymb, latexsym, graphicx, geometry, xcolor, colortbl,  mathtools}

\def\arsinh{{\rm arsinh}\>}

\theoremstyle{plain}
\newtheorem{theorem}{Theorem}[section]
\newtheorem{corollary}[theorem]{Corollary}
\newtheorem{lemma}[theorem]{Lemma}
\newtheorem{proposition}[theorem]{Proposition}
\theoremstyle{definition}
\newtheorem{definition}[theorem]{Definition}
\newtheorem{remarks}[theorem]{Remarks}

\numberwithin{equation}{section}

\usepackage{amsaddr}

\title[Hua-Pickrell diffusions]{ Hua-Pickrell diffusions and differential equations related with pseudo-Jacobi polynomials}

\author{Martin Auer, Michael Voit} 
\address{Fakult\"at Mathematik, Technische Universit\"at Dortmund,
          Vogelpothsweg 87,
          D-44221 Dortmund, Germany}
\email{martin.auer@math.tu-dortmund.de, michael.voit@math.tu-dortmund.de}


\begin{document}
\date{\today}

\begin{abstract} Following Assiotis (2020), we study general $\beta$-Hua-Pickrell diffusions
  of $N$ particles on $\mathbb R$ as solutions of the stochastic differential equations (SDEs)
  $$dX_{j,t}=\sqrt{2(1+X_{j,t}^2)}\,dB_{j,t}+\beta\left[b-a X_{j,t}+\sum_{l=1,\ldots, N; \> l\neq j}\frac{X_{j,t}X_{l,t}+1}{X_{j,t}-X_{l,t}}\right]dt\,,\;\;
  (j=1,\ldots,N)$$
  with $\beta\ge 1,\> a,b\in\mathbb R$. These processes form a subclass of the Pearson diffusions which are defined as
   solutions of algebraic SDEs where the moments of the empirical distributions
  $\mu_t^N:=\frac{1}{N}\sum_{j=1}^N \delta_{X_{j,t}}$
  can be computed inductively.
  This Pearson class also contains other well known diffusions  like
  Dyson Brownian motions, and multivariate Laguerre and Jacobi processes.

  After the time normalization $t\mapsto t/\beta$, the SDEs above  degenerate  in the frozen case for $\beta=\infty$
  into ordinary differential equations
which are related to pseudo-Jacobi polynomials.
For $N\to\infty$ and under suitable initial conditions,
the empirical distributions $\mu_t^N$ converge weakly almost surely for $t>0$ to some limit which is independent from $\beta\in[1,\infty]$.
For $a=-N, b=0$, we describe the limit explicitly via free convolutions. Moreover, if $a=cN$ for some $c>0$, the solutions of our SDEs  converge
for $t\to\infty$ to  stationary distributions, which are  Hua-Pickrell (or Cauchy) measures. We thus obtain connections
between known results for the empirical distributions of these ensembles and the zeros of the pseudo-Jacobi polynomials.
Furthermore, we derive a freezing central limit theorem for $\beta\to\infty$ for  the Hua-Pickrell ensembles which is  related to these zeros. 
\end{abstract}

\maketitle

\section{Introduction}

In \cite{As3}, Assiotis introduced a quite general class of noncolliding interacting particle systems with $N$ particles on $\mathbb R$ or an interval 
$I\subset \mathbb R$. These models are described by diffusion processes $(X_t:=(X_{1,t},\ldots, X_{N,t}))_{t\ge0}$ with $X_{1,t}< X_{2,t}< \ldots< X_{N,t}$
which satisfy the stochastic differential equations (SDE)
\begin{equation}\label{SDE-Pearson}
  dX_{j,t}=\sqrt{2p(X_{j,t})}\,dB_{j,t}+\beta\Bigl( \frac{1}{2}q( X_{j,t}) + \sum_{l=1,\ldots,N; \> l\neq j}\frac{p(X_{j,t})}{X_{j,t}-X_{l,t}}\Bigr)dt
  \quad \quad(j=1,\ldots,N)
\end{equation}
with a Brownian motion  $(B_t:=(B_{1,t},\ldots, B_{N,t}))_{t\ge0}$, and polynomials $p,q$ with $\operatorname{deg}\> p\le 2$ and  $\operatorname{deg}\> q\le 1$.
These diffusions are called $\beta$-Pearson diffusions, as for $N=1$ the stationary distributions of these processes
(whenever they exist) are Pearson distributions; see Section 19.4 of \cite{C} and \cite{Wo, FS} for the statistical background.
Up to suitable changes of coordinates, 
the classical dynamic random matrix models like Dyson Brownian motions, $\beta$-Laguerre processes and $\beta$-Jacobi processes on compact intervals and on
$I=[0,\infty[$ form examples of such $\beta$-Pearson diffusions. Moreover,
    up to suitable changes of coordinates,  the class of $\beta$-Pearson diffusions just consists of these classical
    classes of examples together mainly with one further
    class of diffusions $(X_t)_{t\ge0}$ which depends on 3 parameters $a,b\in\mathbb R$ and $\beta\ge1$, where the processes satisfy the
Hua-Pickrell SDEs
\begin{equation}\label{eq_SDE_Hua-Pickrell}
  dX_{j,t}=\sqrt{2(1+X_{j,t}^2)}\,dB_{j,t}+\beta\left[b-a X_{j,t}+\sum_{l\colon l\neq j}\frac{X_{j,t}X_{l,t}+1}{X_{j,t}-X_{l,t}}\right]dt
  \quad \quad(j=1,\ldots,N)
\end{equation}
on the interior of the  closed Weyl chambers $$C_N^A:=\{x\in\mathbb R^N:\> x_1\le x_2\le\ldots\le x_N\}$$ for $N\ge2$ and $t>0$
with some  initial conditions $X_0\in C_N^A$.
These diffusions $(X_t)_{t\ge0}$ were studied by Assiotis \cite{As1, As2, As3} mainly for $\beta=2$ and named Hua-Pickrell diffusions.
This name and the concrete parametrization in \eqref{eq_SDE_Hua-Pickrell} are motivated by the fact that for $\beta=2$ and suitable $a,b$,
 the diffusions satisfying \eqref{eq_SDE_Hua-Pickrell} admit
unique invariant probability distributions which are the image measures of Hua-Pickrell measures on the $N$-dimensional torus $\mathbb{T}^N$
under the inverse Cayley transform
$$	(\mathbb{T}\setminus\{-1\})^N\to\mathbb{R}^N, z_j\mapsto i\frac{1-z_j}{1+z_j}\,,\;\;j\in\{1,\dots,N\}$$
in all coordinates;  see \cite{BO, BNR, As1, As2} and references therein for these Hua-Pickrell distributions.

Motivated by the classical dynamic random matrix ensembles  above (see e.g.~\cite{AGZ, AV, AVW, VW1}), we now consider the renormalized diffusions
$(\tilde{X}_t:=X_{2t/\beta})_{t\geq0}$ for $\beta\ge1$, which satisfy 
\begin{equation}\label{eq_SDE_Hua-Pickrell-norm}
  d\tilde X_{j,t}=2\sqrt{(1+\tilde X_{j,t}^2)/\beta}\,dB_{j,t}+2\left[b-a\tilde X_{j,t}+
    \sum_{k\colon k\neq j}\frac{\tilde X_{j,t}\tilde X_{k,t}+1}{\tilde X_{j,t}-\tilde X_{k,t}}\right]dt \quad\quad(j=1,\ldots,N)
\end{equation}
with the same initial condition. For $\beta=\infty$, these SDEs
degenerate into the deterministic ODEs
\begin{equation}\label{eq_ODE_Hua-Pickrell}
	\frac{d}{dt}x_{j,t}=2\left[b-a x_{j,t}+\sum_{k\colon k\neq j}\frac{x_{j,t}x_{k,t}+1}{x_{j,t}-x_{k,t}}\right] \quad\quad  (j=1,\ldots,N)
\end{equation}
with the initial condition $\hat{x}_0\in C_N^A$. Solutions of these initial value problems will be called frozen Hua-Pickrell diffusions.

It is known for the classical Pearson diffusions mentioned above for $1\le \beta<\infty$ (see e.g. \cite{AGZ} for Dyson Brownian motions and \cite{GM} in the other classical cases) and  for $\beta=\infty$ (see \cite{VW2, AVW}) as well as for the  Hua-Pickrell diffusions with $\beta=1,2,4$
by \cite{GM, As2} that the corresponding initial value problems admit unique strong solutions.
The same holds for the general Hua-Pickrell cases
\eqref{eq_SDE_Hua-Pickrell-norm} and \eqref{eq_ODE_Hua-Pickrell} for $\beta\in[1,\infty]$;
we shall prove the following result in Section \ref{appendix} below:

\begin{theorem}\label{lem_HP_exist_unique}
	Let $\beta\in[1,\infty] $ and $a,b\in\mathbb R$.
	Then for any initial value $\hat{x}_0\in C_N^{A}$ there exists  a unique strong solution $(X_t)_{t\geq0}$ of
        \eqref{eq_SDE_Hua-Pickrell-norm} (and thus of \eqref{eq_SDE_Hua-Pickrell}) for $\beta<\infty$ and  of \eqref{eq_ODE_Hua-Pickrell}
        for $\beta=\infty$ respectively. In both cases, the solutions are in the interior of  $C_N^{A}$ almost surely for all $t>0$, i.e., the particle systems are non-colliding.
\end{theorem}

In this paper we study several aspects of the solutions of the SDEs \eqref{eq_SDE_Hua-Pickrell-norm} for $\beta\in[1,\infty]$.
Like for the classical dynamic random matrix ensembles it will turn out that in this time renormalized case, the processes have many
features in common whenever the parameter $\beta$ is varying only. For this reason we first study the
deterministic case $\beta=\infty$ in Section 2. In this case, the ODEs \eqref{eq_ODE_Hua-Pickrell} are closely related to the 
pseudo-Jacobi polynomials \cite{Ask1, Ask2, JT, L}. In particular, we give conditions on the parameters $a,b$ under which
\eqref{eq_ODE_Hua-Pickrell} has a stationary solution (see Lemma \ref{lem_HP_pseudo_Jac} and Corollary \ref{lem_pseudo_Jac_char}),
and we show that
in the case of the existence, the stationary solutions of 
\eqref{eq_ODE_Hua-Pickrell} are just the ordered zeros of these polynomials which form
finite sequences of orthogonal polynomials; see Lemma \ref{stat-sol-hua-special}. Furthermore, we show that 
\eqref{eq_ODE_Hua-Pickrell} is equivalent to the fact that the polynomial $H(t,z) := \prod_{j=1}^N (z-x_{j,t})$
solves some inverse Hua-Pickrell-type
heat equation with potential; see Lemma \ref{inverse-heat-hua-pick} below. Moreover, we show that for solutions
of \eqref{eq_ODE_Hua-Pickrell}, some symmetry condition at the initial time $t=0$ is preserved for all $t\ge 0$,
and that thus these symmetric solutions are related to solutions of certain frozen symmetric Heckman-Opdam processes
associated with the root systems A and BC in some way; see Lemmas \ref{symmetry-cond},
\ref{symmetry-squared}, and \ref{symmetry-heckman-a} for the details.

 The results on the existence of stationary solutions of the ODEs \eqref{eq_ODE_Hua-Pickrell} in Section 2
 for $\beta=\infty$  will be extended to the SDEs  \eqref{eq_SDE_Hua-Pickrell-norm}  and \eqref{eq_SDE_Hua-Pickrell} for $\beta\in[1,\infty[$
     in Section 3. It will turn out that for $N\ge2$, $\beta\ge1$, $a\ge -1/\beta$, and $b\in\mathbb R$,
     the SDEs \eqref{eq_SDE_Hua-Pickrell-norm} have unique stationary distributions, the so-called Hua-Pickrell measures
     $\rho_{N,a,b,\beta}$ on $C_N^A$, which have  the Lebesgue densities
\begin{align}\label{eq_Hua_Pick_density-intro}
 f_{N,a,b,\beta}(x):= C_{N,a,b,\beta}\prod_{j=1}^N\left(\left(1+x_j^2\right)^{\frac{\beta}{2}(1-N-a)-1}\exp(\beta b \arctan(x_j))\right)\prod_{1\leq j<k\leq N}\lvert x_j-x_k\rvert^{\beta}\,
\end{align}
with  known  normalizations $C_{N,a,b,\beta}>0$. 

As these connections hold for all $N$ in the same way, we shall transfer known limit theorems for the empirical distributions
with explicit limit distributions for some Heckman-Opdam processes in \cite{AV, Au1, Au2, Ho}
to some Hua-Pickrell processes in Section 5. We shall also show there that these free limits are independent from $\beta\in[1,\infty]$.
To describe these result, we now assume that
 the parameters $ a,b$ in \eqref{eq_SDE_Hua-Pickrell-norm} depend on $N$. More precisely,
for all $N\in\mathbb N$ we  consider the solutions $(X_t^N)_{t\geq0}$
of the SDEs
\begin{equation}\label{eq_SDE_N_rescale-intro}
	dX_{j,t}^N=2\sqrt{\frac{1+\left(X^N_{j,t}\right)^2}{\beta N}}\,dB_{j,t}+\frac{2}{N}\left[b_N-a_NX_{j,t}^N+\sum_{k\colon k\neq j}\frac{X^N_{j,t}X^N_{k,t}+1}{X^N_{j,t}-X^N_{k,t}}\right]dt
\,,\;\; (j=1,\ldots,N)
\end{equation}
on $C_N^A$ with $\beta\in[1,\infty],a_N,b_N\in\mathbb R$.
Please notice that, in comparison to \eqref{eq_SDE_Hua-Pickrell-norm}, we now use the time scaling $t\mapsto t/N$.
The main results for the empirical measures
$$
	\mu_{N,t}:=\frac{1}{N}\sum_{j=1}^N\delta_{X_{j,t}^N} \quad (t\ge0).
$$
        of the processes $(X_t^N)_{t\geq0}$ in Section 5 are  as follows:
        
\begin{theorem}\label{limit_emp_meas-intro}
	Let $\beta\in[1,\infty]$.
	Assume that $\lim_{N\to\infty}a_N/N=:\hat{a}$ and $\lim_{N\to\infty}b_N/N=:\hat{b}$ exist, and that
        there exists a probability measure  $\mu\in M^1(\mathbb{R})$ with
    $$
    	\left\lvert\int_{\mathbb{R}}x^n\,d\mu(x)\right\rvert\leq(\gamma n)^n\;\;\text{for all}\;\;n\in\mathbb{N}\;\;\text{and some}\;\;\gamma>0
    $$
   such that all moments of the empirical measures	$\mu_{N,0}$ at time 0
	converge to those of $\mu$ for $N\to\infty$. 
	Then there exist probability measures $(\mu_t)_{t\geq0}\subset M^1(\mathbb{R})$ with
	$$
		\mu_{N,t}\xrightarrow{N\to\infty}\mu_t\;\;\text{weakly for all}\;\;t\geq0\;\;\text{a.s.}
	$$
	        The  Cauchy transforms $G(t,z):=\int_{\mathbb{R}}\frac{1}{z-x}\,d\mu_t(x)$, $\Re(z)\neq0$, of the $\mu_t$ satisfy the
                partial differential equation (PDE)
	\begin{equation}\label{eq_HP_Cauchy_PDE-intro}
		\partial_tG(t,z)
		=-\partial_z\left[\left(-2(\hat{a}+1)z+2\hat{b}\right)G(t,z)+\left(z^2+1\right)G^2(t,z)\right]\,.
	\end{equation}
        Furthermore, for $\hat{a}>0$, the measures $\mu_{t}$ converge weakly for $t\to\infty$ to the
        equilibrium Hua-Pickrell measures $\mu_{\operatorname{HP},\hat{a},\hat{b}}$
with parameters $\hat a,\hat b$
which are the probability measure on $\mathbb{R}$ with density
\begin{equation}\label{free-limit-densities}
	\mu_{\operatorname{HP},\hat a,\hat b}(dx):=
	\begin{cases}
		\frac{\hat a\sqrt{(x-x_-)(x_+-x)}}{\pi(1+x^2)}dx\,,\quad &x\in(x_-,x_+)\,,\\
		0\,,\quad&\text{otherwise}\,
	\end{cases}
\end{equation}
with
$$
x_{\pm}:=\frac{1}{\hat a^2}\left(-\hat b(\hat a+1)
\pm\sqrt{(2\hat a+1)(\hat a^2+\hat b^2)}\right)\,.
$$
The  equilibrium Hua-Pickrell measures $\mu_{\operatorname{HP},\hat{a},\hat{b}}$  are the weak limits of the emirical distributions of the
 Hua-Pickrell measures
     $\rho_{N,a_N,b_N,\beta}$ on $C_N^A$.
\end{theorem}

We point out that this global limit result fits to the microscopic edge limits in \cite{AM}.

For  $\hat{a}=-1$ and $\hat{b}=0$, we shall 
identify the $\mu_t$ for general starting conditions via an explicit transform of the {\it free positive multiplicative Brownian motion}
studied in \cite{Z} by using the semicircle distribution $\mu_{\operatorname{sc},R}\in M^1(\mathbb R)$ with radius $R>0$ with densities
$$\frac{2}{\pi R^2}\sqrt{R^2-x^2} {\bf 1}_{[-R,R]}$$ together with  the uniform distribution
  $\operatorname{Unif}_{[a,b]}$  on $[a,b]$ with density $\frac{1}{b-a}{\bf 1}_{[a,b]}$.

\begin{theorem}
  Let $(\mu_t)_{t\geq0}$ be the  probability measures from Theorem \ref{limit_emp_meas-intro} with
 $\hat{a}=-1$ and $\hat{b}=0$  such that the starting distribution $\mu$ is even.
  Let $\nu\in M^1((0,\infty))$ be the unique  measure
which is invariant under the pushforward by $x\mapsto x^{-1}$, and with $f(\nu)=\mu$ with
	$$
		f(x):=\frac{x^{1/2}-x^{-1/2}}{2}\,.
	$$
	Then
	$$
	\mu_t
	=f\left(\nu\boxtimes\exp\left(\mu_{\operatorname{sc},4\sqrt{2t}}\boxplus\operatorname{Unif}_{[-4t,4t]}\right)\right)\;\;\text{for all}\;\;t\geq0\,,$$
	where $f(\cdot),\exp(\cdot)$ are the pushforwards under  $f,\exp$,  and where $\boxtimes,\boxplus$ are
        the usual free additive and multiplicative convolutions (see \cite{MS,BV}).
\end{theorem}

 Besides these limits for the empirical measures for $N\to\infty$ we present a freezing central limit theorem (CLT) for $\beta\to\infty$ for
 the Hua-Pickrell measures
     $\rho_{N,a_N,b_N,\beta}$ for  fixed $N$ and $a,b$ in Section 4. This CLT fits to  corresponding freezing CLTs for  the classical random matrix models, i.e.,
 $\beta$-Hermite, Laguerre and Jacobi ensembles in \cite{AHV, DE, GK, He, HV, V1, V2} and is as follows:

\begin{theorem}\label{clt-hua-pickrell-alg-intro}
 	Let $a>0$, $b\in\mathbb R$, $\beta\ge1$, and $N\ge 2$ an integer.
 	Let $ X_{N,a,b,\beta}$ be random variables with the distributions $\rho_{N,a,b,\beta}$. 
 	Then the distributions of $\sqrt{\beta}( X_{N,a,b,\beta}-z)$
 	converge weakly for $\beta\rightarrow\infty$ to the centered $N-$dimensional normal distribution
        $N(0,\Sigma)$ with the positive definite  covariance matrix ${\Sigma}$ whose inverse  
 	$S:=({s}_{j,k})_{j,k=1,...,N}:={\Sigma}^{-1}$ satisfies
 	\begin{align*}
 	 s_{j,k}=
 	\begin{cases}\sum_{ l: \> l\ne j}
 	  \frac{1}{(z_j-z_l)^2}+\frac{(N+a-1)(1-z_j^2)+2bz_j}{(1+z_j^2)^2} &\textit{ for } j=k\\
 	-\frac{1}{(z_j-z_k)^2}&\textit{ for } j\neq k.
 	\end{cases}
 	\end{align*}
\end{theorem}

Further informations about the inverse covariance matrices $S$ will be discussed in Section 4.
 
As already mentioned above, we shall prove  Theorem \ref{lem_HP_exist_unique} in an appendix in Section 6.

Parts of this paper are contained in the thesis \cite{Au1} of the first author of this paper.

\section{The frozen case and pseudo-Jacobi polynomials}

As announced in the introduction, we  study the frozen case $\beta=\infty$ for a fixed number $N\ge2$ in this section, i.e., we consider
 the deterministic ODEs
\begin{equation}\label{eq_ODE_Hua-Pickrell2}
\frac{d}{dt}x_{j,t}=2\left[b-a x_{j,t}+\sum_{k\colon k\neq j}\frac{x_{j,t}x_{k,t}+1}{x_{j,t}-x_{k,t}}\right] \quad\quad  (j=1,\ldots,N)
\end{equation}
on  $C_N^A$. We shall show that the stationary solutions  of \eqref{eq_ODE_Hua-Pickrell2}
are closely related with the zeros of the  pseudo-Jacobi polynomials studied in \cite{Ask1, Ask2, JT, L}.
These polynomials are defined as follows:

\begin{definition} Consider the classical Jacobi polynomials $(P_n^{(\alpha,\beta)})_{n\ge0}$ defined by
 $$P_n^{(\alpha,\beta)}(x):=\frac{(\alpha+1)_n}{n!}\sum_{k=0}^n \frac{(-n)_k(\alpha+\beta+n+1)_k}{(\alpha+1)_k k!}\Bigl(\frac{1-x}{2}\Bigr)^k$$
      with the Pochhammer symbol $(a)_k=a(a+1)\cdots(a+k-1)$ as defined e.g.~in \cite{Sz}. These polynomials are usually
      considered for $\alpha,\beta\in ]-1,\infty[$ in which case they are orthogonal w.r.t.~ the weight function
          $(1-x)^\alpha(1+x)^\beta$ on $]-1,1[$.
          However, these polynomials exist for all $\alpha,\beta\in\mathbb C$ and $n=0,1,\ldots$, where    $P_n(\alpha,\beta)$
          is a polynomial of degree at most $n$.

For parameters $a,b\in\mathbb R$ and $n=0,1,\ldots$, the pseudo-Jacobi polynomials are now defined by 
	\begin{equation}\label{def-pseudo-jac}
		P_n(x;a,b):=(-i)^nP_n^{(a+ib,a-ib)}(ix)
		=(-i)^n\frac{(a+ib+1)_n}{n!}\sum_{k=0}^n\frac{(-n)_k(2a+n+1)_k}{(a+ib+1)_k k!}\left(\frac{1-ix}{2}\right)^k.
	\end{equation}
        These polynomials (partially with different parametrizations)
        are known in the literature also as Romanovski(-Routh) polynomials; see \cite{RWAK, Rom, Rou, AD}
        and references therein. 
\end{definition}

We collect some basic known facts for the  pseudo-Jacobi polynomials:

\begin{lemma}\label{facts-peudojacobi-trivial} Let $a,b\in\mathbb R$ and $n=0,1,\ldots$.
\begin{enumerate}
\item[\rm{(1)}] Three term recurrence:
  \begin{align} (n+1)(n+2a+1)(n+a)P_{n+1}(x;a,b)= &(2n+2a+1)(x(n+a)(n+a+1)+ab)P_n(x;a,b) \notag \\ &+((n+a)^2 +b^2)(n+a+1)P_{n-1}(x;a,b).
    \end{align} 
  with $P_{-1}:=0, P_0=1$.
\item[\rm{(2)}] $P_n(x;a,b)$ is a polynomial with real coefficients with degree at most $n$.
\item[\rm{(3)}]  $P_n(x;a,b)$ is a solution of the ODE
  \begin{equation}\label{eq_pf_HP_ODE_station}
	 	(1+x^2)f''(x)+2[(a+1)x+b]f'(x)-n(n+2a+1)f(x)=0\,,\;\;x\in\mathbb{R}\,.
	 \end{equation}
\item[\rm{(4)}] If $a<0$ and $N<-a$, then  $(P_n(\cdot;a,b))_{n\in\{0,\dots,N\}}$ is a finite sequence of orthogonal polynomials w.r.t.~the
weight function
  \begin{equation}	w(x):=(1-ix)^{a+ib}(1+ix)^{a-ib}=(1+x^2)^ae^{2b\arctan(x)}.\end{equation}
	More precisely,
	\begin{gather*}
	  \int_{\mathbb{R}}P^2_n(x;a,b)w(x)\,dx = \frac{2\pi \Gamma(-n-2a)}{(-2n-2a-1) n! |\Gamma(-n-a+ib)|^2}
          <\infty\;\;\text{for}\;\;n\leq N-1\,,\\
		\int_{\mathbb{R}}P_n(x;a,b)P_m(x;a,b)w(x)\,dx=0 \;\;\text{for}\;\;n<m\leq N\,.
	\end{gather*}
\item[\rm{(5)}] For $a<0$ and $N<-a$, $P_N(\cdot;a,b)$ has $N$ different real zeroes  $x_1<\dots<x_N$.
\end{enumerate}
\end{lemma}

\begin{proof} (1) and (3) follow from the corresponding formulas for the classic
Jacobi polynomials in Ch.~IV of \cite{Sz} and  analytic continuation; see also \cite{JT}. (2) follows from (1).
For (4) see Eq.~(1.10) of \cite{Ask2} or \cite{Ask1, Rom, Rou}, and (5) follows from (4) and classical results on orthogonal polynomials;
see also  \cite{JT}.
\end{proof}

We now are mainly interested in the case $a<0$ and $N<-a$. The following result is
motivated by results on potential theory of  the classical orthogonal polynomials due to Stieltjes;
see Section 6.7 of \cite{Sz} for the background and \cite{V1, HV} for reformulations of these classical results, which fit to the following result:

\begin{lemma}\label{lem_pseudo_Jac_char}
  Let $a<0$, $b\in\mathbb R$, and $N<-a$. Then the following assertions are equivalent for $x=(x_1,\ldots,x_N)\in C_N^A$:
\begin{enumerate}
\item[\rm{(1)}] $x$ is the unique global maximum on $C_N^A$ of the function
  $$h_{N,a,b}(x):= \prod_{j=1}^N \Bigl((1+x_j^2)^{(a+1)/2} e^{ b\arctan(x_j)}\Bigr)\cdot \prod_{j,k:\> 1\le j< k \le N} (x_k-x_j).$$
\item[\rm{(2)}] For all $j=1,\ldots,N$,
 $$	0=(a+1)x_j+b+\sum_{k\colon k\neq j}\frac{1+x_j^2}{x_j-x_k}.$$
\item[\rm{(3)}] $x_1<\dots<x_N$ are the ordered zeros of $P_N(x;a,b)$.
\end{enumerate}
\end{lemma}

\begin{proof} We first check that the function $h_{N,a,b}$ in (1) has a global maximum on  $C_N^A$, which is in the interior of $C_N^A$.
  For this we notice that $h_{N,a,b}>0$  on the interior of $C_N^A$ with $h=0$ on the boundary of $C_N^A$. Furthermore,
 $\prod_{j=1}^N e^{ b\arctan(x_j)}$ is obviously bounded on  $C_N^A$, and it can be easily checked by induction on $N\ge 2$ that 
  $$\prod_{j=1}^N(1+x_j^2)^{-(N-1)/2} \prod_{j,k:\> 1\le j< k \le N} (x_k-x_j)$$
    is  also  bounded on  $C_N^A$. Hence, $N<-a$ ensures that
    $$\lim_{\|x\|\to\infty, \> x\in C_N^A} h_{N,a,b}(x)=0.$$
    These observations and standard arguments on extrema of continuous functions ensure that $h$ has the claimed maximum.

    We now consider some global maximum of $h_{N,a,b}$ on $C_N^A$, say in $x_0$, which is in the interior of $C_N^A$. Then $\nabla \ln h(x_0)=0$.
    A short computation now shows that (2) holds.

    We next check $(2)\Longrightarrow (3)$.  Assume that (2) holds for $(x_1,\ldots,x_N)$ which necessarily is contained in the interior of $C_N^A$.
Define the polynomial	 $f(x):=\prod_{j=1}^N(x-x_j)$. Then
	 $$
	 	f'(x)=\sum_{j=1}^N\prod_{k\colon k\neq j}(x-x_k)\;\;\text{and}\;\;f''(x)=\sum_{j,k\colon j\neq k}\,\prod_{l\colon l\neq j,l\neq k}(x-x_l)\,.
	 $$
	 Hence
	 $$
	 	\frac{f''(x_j)}{f'(x_j)}
	 	=\frac{2\sum_{k\colon k\neq j}\prod_{l\colon l\neq k,l\neq j}(x_j-x_l)}{\prod_{k\colon k\neq j}(x_j-x_k)}
	 	=2\sum_{k\colon k\neq j}\frac{1}{x_j-x_k}  \quad \quad (j=1,\dots,N)
	 $$
	 and thus
	 $$
	 	0=(a+1)x_j+b+\frac{1+x_j^2}{2}\frac{f''(x_j)}{f'(x_j)}    \quad \quad (j=1,\dots,N).
	 $$
	 This means that $x\mapsto 2[(a+1)x+b]f'(x)+(1+x^2)f''(x)$ is a polynomial of degree at most $N$ which vanishes at the points $x_j$.
	 Hence this function is equal to a constant multiple of $f$.
	 Comparing the leading coefficients, we see  that this constant is $N(N+2a+1)$.
	 Thus, $f$ satisfies \eqref{eq_pf_HP_ODE_station}, i.e., 
	 $$	(1+x^2)f''(x)+2[(a+1)x+b]f'(x)-N(N+2a+1)f(x)=0.$$
	 As $P_N(\cdot;a,b)$ is also a solution of \eqref{eq_pf_HP_ODE_station}, it suffices 
	  to show that each polynomial solution of degree at most $N$ of \eqref{eq_pf_HP_ODE_station} must be a constant multiple of $f$.
	For this we rewrite  \eqref{eq_pf_HP_ODE_station} as
	$$
		\frac{d}{dx}\left(w(x)(1+x^2)f'(x)\right)-N(N+2a+1)w(x)f(x)=0\,.
	$$
	Now, let $g$ be another polynomial solution of degree at most $N$ of this equation.
	Then
	$$
		\frac{d}{dx}\left(w(x)(1+x^2)(f'(x)g(x)-f(x)g'(x))\right)=0.
	$$
	Hence, $w(x)(1+x^2)(f'(x)g(x)-f(x)g'(x))$ is constant.
	As the coefficient of $f'g-fg'$ associated with $x^{2N-1}$ disappears, $f'g-fg'$ is a polynomial of degree smaller than $2N-1$.
        As $a<-N$,
 it follows that
	$$
		\lim_{x\to\infty}w(x)(1+x^2)(f'(x)g(x)-f(x)g'(x))=0\,.
	$$
	Thus $f'g-fg'=0$. Hence, $g$ must be a multiple of $f$.
	This completes the proof of $(2)\Longrightarrow (3)$.

        In summary, we conclude that (1), (2), and (3) admit unique solutions, and that the conditions there are equivalent.
\end{proof}

\begin{remarks}\label{remark-ode-solutions-stat1}
\begin{enumerate}
\item[\rm{(1)}] The proof of Lemma \ref{lem_pseudo_Jac_char} implies again for $N<-a$
  that $P_N(\cdot;a,b)$ has $N$ different real zeros for $N<-a$.
\item[\rm{(2)}] If $N\ge -a$, then more complicated conditions are necessary in order to conclude whether
  the conditions in (1), (2), (3) in Lemma  \ref{lem_pseudo_Jac_char} have  solutions. Moreover, for $N\ge2$, the equivalence  $(2) \Longleftrightarrow (3)$
  may fail.

  We give an example with $N=2$, $b=0$. Here, it can be easily seen that (1) and (2) admit real solutions
  precisely for $a<-3/2$, where these solutions are unique and given by $(x_1,x_2)=(-(-2a-3)^{-1/2}, +(-2a-3)^{-1/2})$.
          On the other hand, for $a=-2$ and by Lemma \ref{facts-peudojacobi-trivial}(1), $P_2(x;-2,0)$ is
          a polynomial of degree 1, i.e., (3) does not hold for the solutions of (1), (2). On the other hand, it can be easily checked
          that for $a<-3/2$, $a\ne -2$,  the  pseudo-Jacobi polynomial $P_2(x;a,0)$ has order 2 and the zeros
          $\pm(-2a-3)^{-1/2}$, i.e.,  the assertion of Lemma \ref{lem_pseudo_Jac_char} is correct in this case.

          In the singular case $a=-2, b=0$, the ODE \eqref{eq_pf_HP_ODE_station} is
          of hypergeometric type with singularities $\pm i$ with indices, which differ by integers, and the space of solutions is a 2-dimensional vector space
          consisting of polynomials; see e.g.~Section 4.22 of \cite{Sz}.

          For a further discussion of the existence of solutions of \eqref{eq_pf_HP_ODE_station} see also Remark \ref{remark-ode-solutions-stat2}
 \end{enumerate}
\end{remarks}

We next use Lemma \ref{lem_pseudo_Jac_char} to characterize the unique
stationary solutions of the ODEs \eqref{eq_ODE_Hua-Pickrell2} in some standard setting. Please notice that the meaning of the constant $a$
in \eqref{eq_ODE_Hua-Pickrell2} is  different from the constant $a$ in the  preceding results on the  pseudo-Jacobi polynomials.
We show that in this standard setting, this stationary solution attracts all solutions:

\begin{lemma}\label{lem_HP_pseudo_Jac}
	Let $a,b\in\mathbb{R}$ with $a>0$.
	Let $(x_t)_{t\geq0}$ be the unique solution of \eqref{eq_ODE_Hua-Pickrell2} with an arbitrary initial condition.
	Then  $\lim_{t\to\infty}x_t=z$, where $z\in C_N^A$ is the vector consisting of the  ordered zeros of the pseudo-Jacobi polynomial
        $P_N(\cdot;-(N+a),b)$.
\end{lemma}

\begin{proof} We first notice  that for each stationary solution $x_t\equiv x_0$  of \eqref{eq_ODE_Hua-Pickrell2}
  and all $j=1,\dots,N$,
	$$
		0=b-ax_{j,0}+\sum_{k\colon k\neq j}\frac{x_{j,0}x_{k,0}+1}{x_{j,0}-x_{k,0}}
		=b-(N-1+a)x_{j,0}+\sum_{k\colon k\neq j}\frac{1+x_{j,0}^2}{x_{j,0}-x_{k,0}}\,.$$
	        As $-(N-1+a)-1=-(N+a)<-N$,  Lemma \ref{lem_pseudo_Jac_char} ensures that $x_0$ is the vector
                consisting of the zeros of $P_N(\cdot;-(N+a),b)$.

                Now let  $x_0\in C_N^A$ be an arbitrary starting point and $(x_t)_{t\geq0}$ the associated solution 
of \eqref{eq_ODE_Hua-Pickrell2}. Consider the
elementary symmetric polynomials  $(e_j^N)_{j\in\{1,\dots,N\}}$ and study the  functions
\begin{equation}\label{eq_elem_poly_trafo}
	y_{j,t}:=e_j^{N}(x_t)\;\; (j=1,\ldots,N).
\end{equation}
We now suppress the dependence on $t$ and use the formula
$$y_m=x_ie_{m-1}^{N-1}(x_{\{1,\dots,N\}\setminus\{i\}})+e_m^{N-1}(x_{\{1,\dots,N\}\setminus\{i\}})$$
with  the vector
$x_{\{1,\dots,N\}\setminus\{i\}}\in\mathbb R^{N-1}$  where the $i$-th coordinate is omitted. Then
	\begin{equation}\label{ode-symmetric}
		\frac{d}{dt}y_m
		=\sum_{i=1}^Ne_{m-1}^{N-1}(x_{\{1,\dots,N\}\setminus\{i\}})\frac{d}{dt}x_i\,.
	\end{equation}
We next recall that the elementary symmetric polynomials satisfy
\begin{equation}\label{elem_sym_pol_eq1}
\begin{split}
	\sum_{i=1}^Ne_{k-1}^{N-1}(x_{\{1,\dots,N\}\setminus\{i\}})=(N-k+1)e_{k-1}^N(x)\,,\\
	\sum_{i=1}^Ne_{k-1}^{N-1}(x_{\{1,\dots,N\}\setminus\{i\}})x_i
	=ke_k^N(x),
\end{split}
\end{equation}
\begin{equation*}
	e_{k-1}^{N-1}(x_{\{1,\dots,N\}\setminus\{i\}})
	-e_{k-1}^{N-1}(x_{\{1,\dots,N\}\setminus\{j\}})
	=-(x_i-x_j)e_{k-2}^{N-2}(x_{\{1,\dots,N\}\setminus\{i,j\}})\,,
\end{equation*}
\begin{equation}\label{elem_sym_pol_eq2}
\begin{split}
	&\sum_{\substack{i,j=1\\i\neq j}}^N
		\frac{e_{k-1}^{N-1}(x_{\{1,\dots,N\}\setminus\{i\}})}{x_i-x_j}
	=\sum_{\substack{i,j=1\\i<j}}^N
		\frac{e_{k-1}^{N-1}(x_{\{1,\dots,N\}\setminus\{i\}})
		-e_{k-1}^{N-1}(x_{\{1,\dots,N\}\setminus\{j\}})}{x_i-x_j}\\
	&=-\sum_{\substack{i,j=1\\i<j}}^Ne_{k-2}^{N-2}(x_{\{1,\dots,N\}\setminus\{i,j\}})
		=-\frac{(N-k+2)(N-k+1)}{2}e_{k-2}^N(x)\,,
\end{split}
\end{equation}
and
\begin{equation}\label{elem_sym_pol_eq3}
\begin{split}
	&\sum_{\substack{i,j=1\\i\neq j}}^N
		\frac{e_{k-1}^{N-1}(x_{\{1,\dots,N\}\setminus\{i\}})x_ix_j}{x_i-x_j}
	=\sum_{\substack{i,j=1\\i<j}}^N
		\frac{e_{k-1}^{N-1}(x_{\{1,\dots,N\}\setminus\{i\}})
		-e_{k-1}^{N-1}(x_{\{1,\dots,N\}\setminus\{j\}})}{x_i-x_j}x_ix_j\\
	&=-\sum_{\substack{i,j=1\\i<j}}^Ne_{k-2}^{N-2}(x_{\{1,\dots,N\}\setminus\{i,j\}})
		x_ix_j
	=-\frac{k(k-1)}{2}e_k^N(x).
\end{split}
\end{equation}
These equations and the ODE for $x_t$ yield that
\begin{equation}\label{elementary-ode}
  \frac{d}{dt}y_m=-c(m)y_m+2b(N-m+1)y_{m-1}-(N-m+2)(N-m+1)y_{m-2} \quad\quad (m=1,\ldots,N)\end{equation}
	with $y_{-1}\equiv0$, $y_0\equiv1$ and with $c(m):=m(2a+m-1)>0$.
	Thus
	\begin{align}\label{solutions-elementary-ode}
		y_{m,t}
		=&e^{-c(m)t}\left(\vphantom{\int_0^t}y_{m,0}+\right.\\
		&\left.\int_0^te^{c(m)r}\left(2b(N-m+1)y_{m-1,r}-(N-m+2)(N-m+1)y_{m-2,r}\right)dr\right)\,.
	\notag\end{align}
	It follows by induction on $m$ that the limit $\lim_{t\to\infty}y_{t}$ exists and does not depend on $(y_{1,0},\dots,y_{N,0})$.
	Hence this limit is $(e_1^N(z),\ldots,e_N^N(z))$.
        As the ordered different zeros of a polynomial depend continuously on the coefficients of the polynomial,
	applying the inverse transform to \eqref{eq_elem_poly_trafo} shows the claim.
\end{proof}

\begin{remarks}\label{remark-ode-solutions-stat2}
\begin{enumerate}
\item[\rm{(1)}] Eq.~\eqref{solutions-elementary-ode} and $c(1)\le 0$ for $a\le0$ imply that all stationary solutions of 
the ODEs \eqref{elementary-ode} and thus of \eqref{eq_ODE_Hua-Pickrell2} are not attracting for $a\le0$.
\item[\rm{(2)}] The structure of Eq.~\eqref{elementary-ode} ensures that the ODEs \eqref{elementary-ode} and  \eqref{eq_ODE_Hua-Pickrell2}
have at most one stationary solution for $a\le0$.
\item[\rm{(3)}] Eq.~\eqref{elementary-ode} for $m=2$ implies that for $a=0$, $b\ne 0$, and all $N\ge0$, the ODEs \eqref{elementary-ode} and 
\eqref{eq_ODE_Hua-Pickrell2} have no stationary solutions in contrast to the case $a=b=0$, $N=2$  in Remark
\ref{remark-ode-solutions-stat1}. For general positive  results for $b=0$ and $a\in]-1/2,0]$ we refer to Corollary
          \ref{stat-sol-hua-special} below.
        \end{enumerate}
\end{remarks}

We next turn to a  result which connects arbitrary solutions $(x_{1,t},\ldots,x_{N,t})$
of \eqref{eq_ODE_Hua-Pickrell2} with some   inverse heat equation for the time-dependent 
polynomials $H(z,t):=\prod_{j=1}^N (z-x_{j,t})$ similar to the Jacobi case in Section 4 of \cite{V2}. 

\begin{lemma}\label{inverse-heat-hua-pick}
	Let $a,b\in\mathbb{R}$.
  Let $x(t):=(x_{1,t},\ldots,x_{N,t})\in C_N^A$ be a differentiable function.
      Then  $x(t)$ is a solution of \eqref{eq_ODE_Hua-Pickrell2} 
      if and only if  $H(t,z):=\prod_{j=1}^N (z-x_{j,t})$
      solves the ``inverse Hua-Pickrell-type heat equation with potential''
 \begin{equation}\label{heat-hu-pickrell}
   H_t=  - (1+z^2) H_{zz}+ \Bigl( -2b+2(a+N-1)z\Bigr) H_{z}  -N(2a+N-1)H.
   \end{equation}
 \end{lemma}

We point out that for $a>0$, and for the stationary solution $z$ in Lemma
\ref{lem_HP_pseudo_Jac} consisting of the zeros of $P_N(\cdot;-(N+a),b)$, we have  $H(t,x)=P_N(x;-(N+a),b)$. Therefore,
\eqref{heat-hu-pickrell} generalizes the ODE \eqref{eq_pf_HP_ODE_station}    for $P_N(\cdot ;-(N+a),b)$.

 \begin{proof}
   Assume first that $x(t)$ satisfies  \eqref{eq_ODE_Hua-Pickrell2} and consider 
    $H$ as defined in the theorem. Moreover, for $j\ne k$, define the polynomials
   \begin{equation}\label{def-h-jacobi}  H_j(t,z):=H(t,z)/(z-x_{j,t}), \quad
     H_{j,k}(t,z):=H(t,z)/((z-x_{j,t})(z-x_{k,t})). \end{equation}
   Then \eqref{eq_ODE_Hua-Pickrell2} and 
   \begin{equation}\label{h1-hua-pickrell} 
 \partial_z H(t,z)=\sum_{j=1}^N H_j(t,z), \quad
 \partial_{zz} H(t,z)=\sum_{j,k: \> j\ne k} H_{j,k}(t,z),  \quad\frac{H_j(t,z)-H_k(t,z)}{x_j(t)-x_k(t)}=H_{j,k}(t,z)\end{equation}
lead to
\begin{align}\label{h2-jacobi}
  \partial_t H(t,z)=&    -\sum_{j=1}^N \frac{d}{dt}x_{j,t} \> H_j(t,z)\\
     =& -2b\sum_{j=1}^N H_j(t,z)+2a \sum_{i=1}^N ((x_{j,t} -z)+z)H_j(t,z)
 -2\sum_{j,k: \> j\ne k} \frac{1+x_{j,t}x_{k,t}}{x_{j,t}-x_{k,t}}\> H_j(t,z)\notag\\
 =& -2b\partial_z H(t,z)+2a \Bigl( z\cdot\partial_z H(t,z)-N\cdot H(t,z)\Bigr)\notag\\
&\quad -2\sum_{j,k: \> j\ne k}\frac{H_j(t,z)}{x_{j,t}-x_{k,t}} -2 \sum_{j,k: \> j\ne k}
\frac{H_j(t,z)x_{j,t}x_{k,t}}{x_{j,t}-x_{k,t}}\notag\\
 =& -2b\partial_z H(t,z)+2a z\cdot\partial_z H(t,z)- 2Na \cdot H(t,z)\notag\\
&\quad  -\sum_{j,k: \> j\ne k}  H_{j.k}(t,z)- \sum_{j.k: \> j\ne k} H_{j.k}(t,z)x_{j,t}x_{k,t}\notag
\end{align}
where in the last $=$ in the last two sums, two summands of the LHS correspond to one on the RHS.
The second last sum in the last formula of (\ref{h2-jacobi}) can be treated via (\ref{h1-hua-pickrell}).  For the last sum we 
observe from (\ref{h1-hua-pickrell}) that
\begin{align}\label{h3-jacobi}
\sum_{j,k: \> j\ne k} &H_{j,k}(t,z)x_{j,t}x_{k,t}\\\
=&\sum_{j,k: \> j\ne k}\Bigl( (z-x_{j,t})(z-_{k,t})+z^2 +z((x_{j,t}-z)+(x_{k,t}-z))\Bigr)  H_{j,k}(t,z)\notag\\
=& N(N-1) H(t,z)+ z^2 \partial_{zz}H(t,z)-2(N-1)z\partial_z H(t,z).
\notag\end{align}
  (\ref{h2-jacobi}) and (\ref{h3-jacobi}) now lead to the inverse  heat equation
(\ref{heat-hu-pickrell}).

Now assume that (\ref{heat-hu-pickrell}) holds. As
$$0=\frac{d}{dt} H(t,x_{j,t})= \partial_{t} H(t,x_{j,t})+\partial_{z} H(t,x_{j,t})\cdot \frac{d}{dt}x_{j,t},$$
we obtain
 for $j=1,\ldots,N$ that
  \begin{align}\label{h4-hua-pickrell}
    \frac{d}{dt}x_{j,t}&=- \frac{\partial_{t} H(t,x_{j,t})}{\partial_{z} H(t,x_{j,t})} \\&=
    \frac{ (1+x_{j,t}^2)\partial_{zz} H(t,x_{j,t}) - \Bigl( -b+2(a+N-1)x_{j,t}\Bigr)\partial_{z} H(t,x_{j,t})}{\partial_{z} H(t,x_{j,t})}\notag\\
    &\quad\quad\quad+  \frac{N(2a+N-1)H(t,x_{j,t})}{\partial_{z} H(t,x_{j,t})}\notag
    \\&= \frac{ (1+x_i(t)^2)\partial_{zz} H(t,x_{j,t}) }{ \partial_{z} H(t,x_{j,t})}   -b+2(a+N-1)x_{j,t}
\notag
 \end{align}
  We now write $H(t,z)=H_j(t,z)(z-x_{j,t})$ for $j=1,\ldots,N$. Hence, as
  $$\frac{\partial_{zz}H(t,x_{j,t})}{\partial_{z}H(t,x_{j,t})}= \frac{2\partial_{z} H_i(t,x_{j,t})}{H_i(t,x_{j,t})},$$
  we get
 \begin{equation}\label{log-derivative-jacobi}
   \frac{\partial_{zz} H(t,x_{j,t})}{\partial_{z} H(t,x_{j,t})}=   \frac{2\partial_{z} H_i(t,x_{j,t})}{ H_i(t,x_{j,t})}=
  2 \sum_{k:k\ne j} \frac{1}{x_{j,t}-x_{k,t}}. \end{equation}
 As
 $$  \frac{1-x_{j,t}^2}{x_{j,t}-x_{k,t}}=  \frac{1-x_{j,t}x_{k,t}}{x_{j,t}-x_{k,t}} - x_{j,t},$$
 we obtain from (\ref{h4-hua-pickrell}) and (\ref{log-derivative-jacobi}) that
 \eqref{eq_ODE_Hua-Pickrell2}  holds.
 \end{proof}
	
 We next discuss some symmetry property for the ODE \eqref{eq_ODE_Hua-Pickrell2} which will be useful below.

 \begin{lemma}\label{symmetry-cond} Let $a\in\mathbb R$ and $b=0$. Let $x_t:=(x_{1,t},\ldots,x_{N,t})$ be a solution of 
   \eqref{eq_ODE_Hua-Pickrell2} with the symmetric initial condition
   $x_{N+1-j,0}=-x_{j,0}$ for $j=1,\ldots,N$. Then for all $t\ge0$,  $x_{N+1-j,t}=-x_{j,t}$ for $j=1,\ldots,N$.
\end{lemma}

 \begin{proof} $\tilde x_t:=(-x_{N,t},\ldots, -x_{1,t})$ also solves \eqref{eq_ODE_Hua-Pickrell2}  with the same initial condition.
   As this initial value problem has a unique solution in $C_N^A$ even for start in $\partial C_N^A$ by Theorem \ref{lem_HP_exist_unique}, the Lemma follows.
 \end{proof}

 In the symmetric setting of Lemma \ref{symmetry-cond} we can transform the Hua-Pickrell-type ODEs  \eqref{eq_ODE_Hua-Pickrell2} into  ODEs
 which are related with  Heckman-Opdam hypergeometric functions of type BC; see \cite{HO, HS} for the background.
 These ODEs appear, for instance, in Sections 5 and 6  in \cite{AVW}.
 For this we state the following result:

\begin{lemma}\label{symmetry-squared} Let $a\in\mathbb R$ and $b=0$. Let $x(t):=(x_{1,t},\ldots,x_{N,t})$ be a  symmetric solution of 
  \eqref{eq_ODE_Hua-Pickrell2} as in Lemma \ref{symmetry-cond} and set $\tilde x_{j,t}:=x_{j,t}^2+1$ for $j=1,\ldots,N$.
\begin{enumerate}
\item[\rm{(1)}]  If $N\ge2$ is even, then $(\tilde x_{1,t},\ldots,\tilde x_{N/2,t})$ satisfies
  \begin{equation}\label{squared-odes-even}
\frac{d}{dt} \tilde x_{j,t}= 2(N+2a)-2(2a+1)\tilde x_{j,t} +4\sum_{k=1,\ldots,N/2,\> k\ne j}
\frac{2 \tilde x_{j,t}\tilde x_{k,t} -\tilde x_{j,t}-\tilde x_{k,t}}{\tilde x_{j,t}-\tilde x_{k,t}} \quad (j=1,\ldots,N/2, \> t\ge0).
\end{equation}
\item[\rm{(2)}]  If $N\ge3$ is odd, then $\tilde x_{(N+1)/2,t}=0$, and  $(\tilde x_{1,t},\ldots,\tilde x_{(N-1)/2,t})$ satisfies
  \begin{equation}\label{squared-odes-odd} \frac{d}{dt} \tilde x_{j,t}= 2(N+2a+1)-2(2a+1)\tilde x_{j,t} +4\sum_{k=1,\ldots,(N-1)/2,\> k\ne j}
  \frac{2 \tilde x_{j,t}\tilde x_{k,t} -\tilde x_{j,t}-\tilde x_{k,t}}{\tilde x_{j,t}-\tilde x_{k,t}}\end{equation}
  for $j=1,\ldots,(N-1)/2, \> t\ge0$.
\end{enumerate}
\end{lemma}

\begin{proof}
  Assume first that $N$ is even. Then, for $j=1,\ldots,N/2$ by \eqref{eq_ODE_Hua-Pickrell2},
  \begin{align}\label{transform-frozen-hua-jacobi}
    \frac{d}{dt} x_{j,t}^2&= 2 x_{j,t}\Biggl(-2a  x_{j,t}+2  \sum_{k=1,\ldots, N,\> k\ne j} \frac{1+x_{j,t}x_{k,t}}{ x_{j,t}- x_{k,t}} \Biggr) \notag\\
    &= -4a x_{j,t}^2 +2(1- x_{j,t}^2)+ 4x_{j,t} \sum_{k=1,\ldots, N/2,\> k\ne j}
    \Biggl( \frac{1+x_{j,t}x_{k,t}}{ x_{j,t}- x_{k,t}} +  \frac{1-x_{j,t}x_{k,t}}{ x_{j,t}+ x_{k,t}}\Biggr) \notag\\
    &= 2-2(2a+1)x_{j,t}^2+ 4 \sum_{k=1,\ldots, N/2,\> k\ne j} \frac{ 2x_{j,t}^2+ 2x_{j,t}^2x_{k,t}^2}{ x_{j,t}^2- x_{k,t}^2} \notag\\
    &=2(N-1)-2(2a+1)x_{j,t}^2+ 4 \sum_{k=1,\ldots, N/2,\> k\ne j} \frac{ x_{j,t}^2+x_{k,t}^2+ 2x_{j,t}^2x_{k,t}^2}{ x_{j,t}^2- x_{k,t}^2} \notag\\
     &= 2(N+2a)-2(2a+1)\tilde x_{j,t}+4\sum_{k=1,\ldots,N/2,\> k\ne j} \frac{2 \tilde x_{j,t}\tilde x_{k,t} -\tilde x_{j,t}-\tilde x_{k,t}}{\tilde x_{j,t}-\tilde x_{k,t}}
 \end{align}
 as claimed. If $N$ is odd, then $x_{(N+1)/2,t}=0$ by symmetry, and for $j=1,\ldots, (N-1)/2$,
 we can use a similar computation as in the even case.
 \end{proof}
 
We next state a further symmetry result which relates frozen Hua-Pickrell diffusions, i.e., solutions of
 \eqref{eq_ODE_Hua-Pickrell2}, 
 with  frozen Heckman-Opdam processes of type $A$. These processes were studied for finite $\beta>0$ in \cite{Sch1, Sch2}. For the analytic background we refer to \cite{HO, HS}.
 After a time normalization as in the preceding cases (see also \cite{AVW}), these processes  $(Z_t)_{t\geq0}$ can be regarded as unique solutions
 of  the SDEs
 $$
 	dZ_{j,t}=\sqrt{\frac{2}{\beta}}\,dB_{j,t}+\sum_{k\colon k\neq j}\coth(Z_{j,t}-Z_{k,t})\,dt \quad (j=1,\dots,N)
 $$
        with initial condition  $Z_0\in C_N^{A}$ for $\beta\in[1,\infty]$; see in particular Section  6.6 of \cite{GM}
        for the existence of solutions of these SDEs.
 As before we now focus on the case $\beta=\infty$  where the diffusion term disappears.
In this case, the transformed deterministic processes  $$(y_t:=(\exp(2Z_{1,t/2}),\dots,\exp(2Z_{N,t/2})))_{t\ge0}$$
then are the unique solutions of the ODEs
 \begin{equation}\label{eq_froz_mult_BM_ODE}
 	\frac{d}{dt}y_{j,t}=-(N-1)y_{j,t}+2\sum_{k\colon k\neq j}\frac{y_{j,t}^2}{y_{j,t}-y_{k,t}} \quad (j=1,\dots,N)
 \end{equation}
 on $C_N^{A}\cap(0,\infty)^N$ for the initial conditions $y_0\in C_N^{A}\cap(0,\infty)^N$.
 Under a symmetry assumption on the starting conditions these solutions are closely related to symmetric solutions of
 \eqref{eq_ODE_Hua-Pickrell2}:
 
 \begin{lemma}\label{symmetry-heckman-a}
 	Let $(y_t)_{t\geq0}$ be the unique solution of \eqref{eq_froz_mult_BM_ODE} in $C_N^{A}\cap(0,\infty)^N$  such that the initial condition 
 	satisfies the symmetry $y_{j,0}=y_{N+1-j,0}^{-1}$ for $j=1,\dots,N$.
 	For $f(y):=(y^{1/2}-y^{-1/2})/2$, consider $x_t:=(f(y_{1,4t}),\dots,f(y_{N,4t}))$ for $t\geq0$.
 	Then:
 	\begin{enumerate}
 		\item $y_{j,t}=y_{N+1-j,t}^{-1}$ for  $j=1,\dots,N$ and $t>0$.
 		\item $(x_t)_{t\geq0}$ is the unique solution   of \eqref{eq_ODE_Hua-Pickrell2} in $C_N^{A}$ for the parameters
                  $a=-N+\frac{1}{2}$, $b=0$. Moreover,
 		$(x_t)_{t\geq0}$ satisfies the symmetry stated in Lemma \ref{symmetry-cond}.
 	\end{enumerate}
 \end{lemma}
 
 \begin{proof}
 	For the proof of (1)
 	assume that $y_{j,0}=y_{N+1-j,0}^{-1}$ for  $j=1,\dots,N$.
 	If we put $\tilde{y}_{j,t}:=y_{N+1-j,t}^{-1}$, then
 	\begin{align*}
 		\frac{d}{dt}\tilde{y}_{j,t}
 		=&(N-1)\tilde{y}_{j,t}+2\sum_{k\colon k\neq j}\frac{\tilde{y}_{j,t}\tilde{y}_{k,t}}{\tilde{y}_{j,t}-\tilde{y}_{k,t}}
 		=-(N-1)\tilde{y}_{j,t}+2\sum_{k\colon k\neq j}\frac{\tilde{y}_{j,t}^2}{\tilde{y}_{j,t}-\tilde{y}_{k,t}}\,.
 	\end{align*}
 	As solutions of   \eqref{eq_froz_mult_BM_ODE} are unique, the claim in (1) follows.\\
 	For (2), we first assume that  $N$ is even. We now use the function
 $$g(y):=(y+y^{-1}+2)/4=1+f^2(y)$$ and 
 transform the solution $(y_t)_{t\geq0}$ of \eqref{eq_froz_mult_BM_ODE} via $\tilde{x}_t:=(g(y_{1,t}),\dots,g(y_{N/2,t}))$.
 	Using (1) we get for $j=1,\dots,N/2$ that
 	\begin{align*}
 		\frac{d}{dt}\tilde{x}_{j,t}
 		=&\frac{y_{j,t}-y_{j,t}^{-1}}{4}\left(-N+1+2y_j\sum_{k=1,\dots,N/2,\,k\neq j}\left(\frac{1}{y_{j,t}-y_{k,t}}+\frac{1}{y_{j,t}-y_{k,t}^{-1}}\right)\right)+\frac{y_{j,t}}{2}\\
 		=&\frac{y_{j,t}-y_{j,t}^{-1}}{4}\left(-N+1+\sum_{k=1,\dots,N/2,\,k\neq j}\frac{y_{j,t}+1-2\tilde{x}_{k,t}}{\tilde{x}_{j,t}-\tilde{x}_{k,t}}\right)+\frac{y_{j,t}}{2}\\
 		=&\frac{y_{j,t}-y_{j,t}^{-1}}{4}\left(-1+\sum_{k=1,\dots,N/2,\,k\neq j}\frac{y_{j,t}+1-2\tilde{x}_{j,t}}{\tilde{x}_{j,t}-\tilde{x}_{k,t}}\right)+\frac{y_{j,t}}{2}\\
 		=&\tilde{x}_{j,t}-\frac{1}{2}+2\sum_{k=1,\dots,N/2,\,k\neq j}\frac{\tilde{x}_{j,t}(\tilde{x}_{j,t}-1)}{\tilde{x}_{j,t}-\tilde{x}_{k,t}}\\
 		=&(N-1)\tilde{x}_{j,t}-\frac{N-1}{2}+\sum_{k=1,\dots,N/2,\,k\neq j}\frac{2\tilde{x}_{j,t}\tilde{x}_{k,t}-\tilde{x}_{j,t}-\tilde{x}_{k,t}}{\tilde{x}_{j,t}-\tilde{x}_{k,t}}\,.
 	\end{align*}
 	This and Lemma \ref{symmetry-squared}(1) now lead to assertion (2) for even $N$.
 	The proof for odd $N$ is  similar, where one has to take into account the correction due to $y_{(N+1)/2,t}=1$ for  $t\geq0$.
 \end{proof}

 We now return to  Lemma \ref{symmetry-squared}.
The ODEs \eqref{squared-odes-odd} and  \eqref{squared-odes-even}
 there are examples of the more general ODEs
\begin{equation}\label{eq_Fisher_Snedecor_ODE}
	\frac{d}{dt}x_{j,t}
	=-p+(p+q)x_{j,t}+\sum_{k\colon k\neq j}\frac{2x_{j,t}x_{k,t}-x_{j,t}-x_{k,t}}{x_{j,t}-x_{k,t}}\quad ( j=1,\dots,N)
\end{equation}
with parameters $p,q\in\mathbb{R}$. These ODEs  are related to
 non-compact Jacobi processes on the Weyl chambers
$C_N^{BC}:=\{x\in\mathbb :\> 1\le x_1\le\ldots\le x_N\}$ studied in \cite{AVW}. We here have:

\begin{lemma}\label{lem_non-col_Jac_frozen_stat}
  Let $q>N-1$, $p<-q$ and let $\hat{x}_0\in C_N^{BC}$. Then, \eqref{eq_Fisher_Snedecor_ODE} has a unique solution
  $(x_t)_{t\geq0}$ with $x_0=\hat{x}_0$ such that $x_t$ is in the interior of $ C_N^{BC}$ for $t>0$.
  
  Moreover, if $(z_1,\dots,z_N)$ is the unique vector in the interior of $ C_N^{BC}$ such that
 $$(2z_N^{-1}-1,\dots,2z_1^{-1}-1)$$ consists of the ordered zeros of the Jacobi polynomial $P_N^{(q-N,-(p+q+1))}$, then
  $(z_1,\dots,z_N)=\lim_{t\to\infty}x_t$ holds for all initial values $\hat{x}_0\in C_N^{BC}$. In particular,
  $(z_1,\dots,z_N)$ is the only  stationary solution of \eqref{eq_Fisher_Snedecor_ODE},
  and this solution is attracting.
\end{lemma}

\begin{proof} For the existence of solutions of   \eqref{eq_Fisher_Snedecor_ODE} we refer to \cite{AVW}.

We next consider a stationary solution   $(z_1,\dots,z_N)$ of \eqref{eq_Fisher_Snedecor_ODE} with $1<z_1<\dots<z_N$. Then
  	$$
		0
		=-p+(p+q)z_j+\sum_{k\colon k\neq j}\frac{2z_jz_k-(z_j+z_k)}{z_j-z_k}\,,\;\;j\in\{1,\dots,N\}\,.
	$$
	        If we  put $y_j:=2z_j^{-1}-1$ and divide this equation by
 $-z_j$, we obtain
	\begin{align}
		0=&pz_j^{-1}-(p+q)+\sum_{k\colon k\neq j}\frac{1+z_kz_j^{-1}-2z_k}{z_j-z_k} \notag\\
		=&(p+(N-1))z_j^{-1}-(p+q)+2\sum_{k\colon k\neq j}\frac{z_j^{-1}(1-z_j^{-1})}{z_j^{-1}-z_k^{-1}} \notag\\
		=&\frac{1}{2}\left[-p-2q+N-1+(p+N-1)y_j+2\sum_{k\colon k\neq j}\frac{1-y_j^2}{y_j-y_k}\right] \notag\\
		=&\frac{1-y_j^2}{2}\left[\frac{q-(N-1)}{y_j-1}-\frac{p+q}{y_j+1}+2\sum_{k\colon k\neq j}\frac{1}{y_j-y_k}\right].
	\end{align}
	By Theorem 6.7.1 of \cite{Sz}, it is known that for $q-N>-1$ and $p+q<0$ these equations in $(y_1,\ldots,y_N)$ have
        a  unique ordered solution, namely  the ordered zeros of the Jacobi polynomial $P_N^{(q-N,-(p+q+1))}$.
        
        Now let $(x_t)_{t\ge0}$ be an arbitrary solution of \eqref{eq_Fisher_Snedecor_ODE}. We
        now proceed as in the proof of Lemma \ref{lem_HP_pseudo_Jac}, and use the elementary symmetric polynomials and
the homeomorphism $$f\colon C_N^A\to f(C_N^A),\> x\mapsto(e_1^N(x),\dots,e_N^N(x)).$$
	Set $y_{t}:=f(x_t)$.
	Using
	$$
		e_m^{N}(x)=x_ie_{m-1}^{N-1}(x_{\{1,\dots,N\}\setminus\{i\}})+e_m^{N-1}(x_{\{1,\dots,N\}\setminus\{i\}}) \quad (x\in\mathbb{R}^N, \> i=1,\dots,N),
	$$
	we obtain (by suppressing the dependence of $t$) that
	\begin{equation*}
	\begin{split}
		\frac{d}{dt}y_m
		=&\sum_{i=1}^Ne_{m-1}^{N-1}(x_{\{1,\dots,N\}\setminus\{i\}})\frac{d}{dt}x_i\,.
	\end{split}
	\end{equation*}
	The ODE for $x_i$, \eqref{elem_sym_pol_eq1}, \eqref{elem_sym_pol_eq3}, and
\begin{equation}\label{elem_sym_pol_eq4}
	\sum_{i,j=1,\ldots,N, \> i<j}(x_i+x_j)e_{k-2}^N(x_{\{1,\dots,N\}\setminus\{i,j\}})
	=(k-1)(N-k+1)e_{k-1}^N(x)\,,\end{equation}
        thus lead to
	\begin{equation}\label{eq_Fisher_Snedecor_ODE_elem_poly}
		\frac{d}{dt}y_{m,t}
		=m(p+q+1-m)y_{m,t}-(p+m-1)(N-m+1)y_{m-1,t}\quad (m=1,\dots,N)
	\end{equation}
	 with the convention $y_{0,t}\equiv0$.
	Thus
	$$
		y_{m,t}
		=e^{m(p+q+1-m)t}\left(y_{m,0}-(p+m-1)(N-m+1)\int_0^te^{-m(p+q+1-m)r}y_{m-1,r}\,dr\right)\,.
	$$
	Since $p+q<0$, we have $m(p+q+1-m)<0$ for all $m\in\{1,\dots,N\}$.
	Inductively one sees that the limit $\lim_{t\to\infty}y_{t}$ exists and does not depend on $(y_{1,0},\dots,y_{N,0})$.
	Hence this limit is equal to $f(z)$ with  the stationary solution  $z$ above.
	A transfer back via $f^{-1}$ shows the claim.
\end{proof}

If we rewrite the two equations in Lemma \ref{symmetry-squared} as equations for $(\tilde x_{t/4})_{t\ge0}$ and compare them with 
\eqref{eq_Fisher_Snedecor_ODE} for $p+q=-(a+1/2)$ and $q=N\pm 1/2$ (with $N$ as in Lemma 
\ref{lem_non-col_Jac_frozen_stat}), we see that Lemma \ref{lem_non-col_Jac_frozen_stat} and 
\eqref{transform-frozen-hua-jacobi} yield the existence of a stationary solution of \eqref{eq_ODE_Hua-Pickrell2}
for $a\in]-1/2,0]$, $b=0$.
    In summary, this and Remark \ref{remark-ode-solutions-stat2}(2) lead to:

    \begin{corollary}\label{stat-sol-hua-special}
      Let $a\in]-1/2,0]$ and $b=0$. Then \eqref{eq_ODE_Hua-Pickrell2} has a unique symmetric stationary solution.
          This solution attracts all symmetric solutions of \eqref{eq_ODE_Hua-Pickrell2}, where this is not correct
          for arbitrary solutions of \eqref{eq_ODE_Hua-Pickrell2}.
\end{corollary}

    Furthermore, for $a>0$,  Lemmas  \ref{symmetry-squared} and  \ref{lem_non-col_Jac_frozen_stat}
    lead to the following observation:

    \begin{corollary}\label{zeros-pseudo-jac-symm} Let $a>0$ and $N\ge2$.
      \begin{enumerate}
\item[\rm{(1)}]     If $N$ is even, and if $\pm z_1,\ldots, \pm z_{N/2}$ are the different zeros of $P_N(x; -N-a,0)$ (which are $\ne0$), then
      $\frac{1-z_j^2}{1+z_j^2}$ ($j=1,\ldots,N/2$) are the zeros of the Jacobi polynomial $P_{N/2}^{(-1/2, a-1/2)}$.
\item[\rm{(2)}]
     If $N$ is odd, and if $\pm z_1,\ldots, \pm z_{(N-1)/2},0$ are the different zeros of $P_N(x; -N-a,0)$, then
     $\frac{1-z_j^2}{1+z_j^2}$ ($j=1,\ldots,(N-1)/2$) are the zeros of  $P_{(N-1)/2}^{(1/2, a-1/2)}$.
     \end{enumerate}
    \end{corollary}


\section{Stationary distributions of Hua-Pickrell diffusions}

In this section we extend Lemma \ref{lem_HP_pseudo_Jac} in the deterministic case to the Hua-Pickrell diffusions $(X_t)_{t\ge 0}$
on $C_N^A$ which satisfy the SDEs
\eqref{eq_SDE_Hua-Pickrell} or their renormalized versions in \eqref{eq_SDE_Hua-Pickrell-norm}. We show
that for $\beta\ge1$ and $a>-1/\beta$, the solutions $X_t$ tend in distribution  for $t\to\infty$ to corresponding $\beta$-Hua-Pickrell probability
measures on $C_N^A$. The associated random
ensembles are referred to in the literature as $\beta$-Hua-Pickrell ensembles as well as as $\beta$-Cauchy ensembles;
see \cite{AD, As1, As2, As3, AGS, F, FR}. There does not seem to exist a  predominant definition of
these measures and their parameters in the literature. We thus use the following definition which fits to the SDEs
\eqref{eq_SDE_Hua-Pickrell} and \eqref{eq_SDE_Hua-Pickrell-norm}:

\begin{definition}\label{def_HP_meas}
  Let $N\ge2, \beta\ge1$ and $a,b\in\mathbb R$ with $a>-1/\beta$.
  The Hua-Pickrell probability
measure $\rho_{N,a,b,\beta}$   is  defined as the probability measure on $C_N^A$  with the Lebesgue density
\begin{align}\label{eq_Hua_Pick_density}
 f_{N,a,b,\beta}(x):= C_{N,a,b,\beta}\prod_{j=1}^N\left(\left(1+x_j^2\right)^{\frac{\beta}{2}(1-N-a)-1}\exp(\beta b \arctan(x_j))\right)\prod_{1\leq j<k\leq N}\lvert x_j-x_k\rvert^{\beta}\,
\end{align}
with the normalization
\begin{align}\label{eq_Hua_Pick_density-norming}
  C_{N,a,b,\beta}^{-1}=& \frac{\pi^N 2^{-\beta N(a+(N-1)/2)}}{N!}\cdot\\
  &\cdot \prod_{j=0}^N \frac{\Gamma(\beta(j+a)+1)\cdot \Gamma((\beta(j+1)+1)}{\Gamma(\beta(j+\frac{a+ib}{2})+1)\cdot
\Gamma(\beta(j+\frac{a-ib}{2})+1)\cdot \Gamma(\beta+1)}.
  \notag
 \end{align}
This constant follows from Exercise 4.7(4) of p.179 and Eq.~(4.4) in \cite{F} and elementary calculus where one has to take the support $C_N^A$
of the measures $\rho_{N,a,b,\beta}$ into account.
\end{definition}

We notice that the Hua-Pickrell densities in \eqref{eq_Hua_Pick_density}  are closely related to the function $h$ in Lemma \ref{lem_pseudo_Jac_char} on the pseudo-Jacobi polynomials
(with the meaning of the constant $a$ there).

The following result is the main result of this section. For the case $\beta=2$ this result is shown in Proposition 4.4 of \cite{As2}.

\begin{proposition}\label{lem_Hua-Pickrell_invariant_meas}
  Let $\beta\ge1$ and $a,b\in\mathbb R$ with $a>-1/\beta$.
  Then the SDE \eqref{eq_SDE_Hua-Pickrell} has $\rho_{N,a,b,\beta}$  as its unique invariant measure.  
\end{proposition}

\begin{proof}
Let $(X_t)_{t\geq0}$ be a solution of \eqref{eq_SDE_Hua-Pickrell}.
We define $(Y_t)_{t\geq0}$ by  putting $Y_{j,t}:=\arsinh(X_{j,t})$ for $j\in\{1,\dots,N\}$, $t\geq0$.
By It\^{o}'s formula and some elementary calculations, $(Y_t)_{t\geq0}$ satisfies
\begin{align*}
	dY_{j,t}
	=&\sqrt{2}\,dB_{j,t}+\beta\left[\left(1-N-\frac{1}{\beta}-a\right)\tanh(Y_{j,t})+b \frac{1}{\cosh(Y_{j,t})}\right.\\
	&\hphantom{\sqrt{2}\,dB_{j,t}+\beta}\left.+\cosh(Y_{j,t})\sum_{k\colon k\neq j}\frac{1}{\sinh(Y_{j,t})-\sinh(Y_{k,t})}\right]dt\,.
\end{align*}
Using 
\begin{align*}
	\frac{\cosh(x)}{\sinh(x)-\sinh(y)}
	=&\frac{1}{2}\left(\frac{\cosh(x)+\cosh(y)}{\sinh(x)-\sinh(y)}+\frac{\cosh(x)-\cosh(y)}{\sinh(x)-\sinh(y)}\right)\\
	=&\frac{1}{2}\left(\coth\left(\frac{x-y}{2}\right)+\tanh\left(\frac{x+y}{2}\right)\right)\,,
\end{align*}
we  rewrite our SDE as
\begin{equation}\label{eq_SDE_Hua_Pickrell_hyp}
\begin{split}
	dY_{j,t}
	=&\sqrt{2}\,dB_{j,t}+\beta\left[\left(1-N-\frac{1}{\beta}-a\right)\tanh(Y_{j,t})+\frac{b}{\cosh(Y_{j,t})}\right.\\
	&\hphantom{\sqrt{2}\,dB_{j,t}+\beta}\left.+\frac{1}{2}\sum_{k\colon k\neq j}\left(\coth\left(\frac{Y_{j,t}-Y_{k,t}}{2}\right)+\tanh\left(\frac{Y_{j,t}+Y_{k,t}}{2}\right)\right)\right]dt\,.
\end{split}
\end{equation}
The corresponding generator is thus given by
\begin{align*}
	\mathcal{L}
	:=\Delta+(-\nabla V(x))\nabla
\end{align*}
with
\begin{equation}\label{eq_hua_pickrell_potential_def}
\begin{split}
	V(x)
	:=&-\beta\left[\left(1-N-\frac{1}{\beta}-a\right)\sum_{j=1}^{N}\log(\cosh(x_j))+b\sum_{j=1}^{N}\arctan(\sinh(x_j))\right.\\
	&\left.\hphantom{-\beta}+\sum_{1\leq j<k\leq N}\log\left(\sinh\left(\frac{x_k-x_j}{2}\right)\cosh\left(\frac{x_j+x_k}{2}\right)\right)\right]\,.
\end{split}
\end{equation}
The candidate for the invariant measure for such an overdamped Langevin dynamic is  the Gibbs-measure $\tilde{\rho}$
with Lebesgue density $\frac{1}{\tilde{Z}}e^{-V(x)}$ with normalization $\tilde{Z}:=\int_{C_N}e^{-V(x)}\,dx$.
In order to check $\tilde{Z}<\infty$, we return to the original coordinates.
We thus determine the image measure of the measure with density
\begin{equation}\label{eq_density_fisher_snedecor_hyp}
\begin{split}
	e^{-V(x)}
	=&\prod_{j=1}^N\left(\left(\cosh(x_j)\right)^{\beta(1-N-1/\beta-a)}\exp\left(\beta b\arctan(\sinh(x_j))\right)\right)\\
	&\times\prod_{1\leq j<k\leq N}\left(\sinh\left(\frac{x_k-x_j}{2}\right)\cosh\left(\frac{x_j+x_k}{2}\right)\right)^{\beta}
\end{split}
\end{equation}
under the map $C_N^{A}\to C_N^{A}$, $x_j\mapsto\sinh(x_j)$, $j\in\{1,\dots,N\}$.
Using $\cosh(\arsinh(x))=\sqrt{1+x^2}$ and
$$
	\sinh\left(\frac{\arsinh(x_i)-\arsinh(x_j)}{2}\right)\cosh\left(\frac{\arsinh(x_i)+\arsinh(x_j)}{2}\right)=\frac{x_i-x_j}{2}\,,
$$
we get a density which is proportional to \eqref{eq_Hua_Pick_density}.
Under our assumptions in the Proposition, this measure is finite  by Definition \ref{def_HP_meas}.
Hence, the Gibbs measure $\tilde{\rho}$ is the unique invariant probability measure of the dynamic \eqref{eq_SDE_Hua_Pickrell_hyp}, see e.g.\ the discussion at the beginning of Section 2 of the survey \cite{MV}.
The claim follows after applying the coordinate change  $x_j\mapsto\sinh(x_j)$, $j\in\{1,\dots,N\}$, as above.
\end{proof}


\section{A central limit theorem in the freezing case $\beta\to\infty$}

In this section we derive a central limit theorem (CLT) for $\beta\to\infty$
for the Hua-Pickrell measures $\rho_{N,a,b,\beta} $ on $C_N^A$ for fixed $N, a>0$, and $b\in\mathbb R$.
For this we first observe that by the definition of the functions $h_{N,a,b}$ in Lemma \ref{lem_pseudo_Jac_char}, the 
Hua-Pickrell measures can be written as
\begin{equation}\label{rho-by-h}
  d\rho_{N,a,b,\beta}(x)=c_{N,a,b,\beta} (h_{N,-N-a,b}(x))^\beta \cdot \prod_{j=1}^N (1+x_j^2)^{-1} dx\end{equation}
  where by Lemma \ref{lem_pseudo_Jac_char}, $h_{N,-N-n,b}$ has a unique maximum on $C_N^A$ in the vector $z=(z_1,\ldots,z_N)$
  consisting of the $N$ different zeroes of $P_N(x; -(N+a),b)$. 
  This implies readily that for $\beta\to\infty$,
  \begin{equation}\label{LLN-Hua-pickrell}
    \rho_{N,a,b,\beta}\to \delta_z \quad\quad\text{weakly.}
  \end{equation}
  Motivated by corresponding CLTs for the classical $\beta$-ensembles (Hermite, Laguerre, Jacobi) in \cite{DE, GK, V1, HV}, we have
the following CLT associated with \eqref{LLN-Hua-pickrell}:

\begin{theorem}\label{clt-hua-pickrell-alg}
 	Let $a>0$, $b\in\mathbb R$, $\beta\ge1$, and $N\ge 2$ an integer.
 	Let $ X_{N,a,b,\beta}$ be random variables with the distributions $\rho_{N,a,b,\beta}$. 
 	Then the distributions of $\sqrt{\beta}( X_{N,a,b,\beta}-z)$
 	converge weakly for $\beta\rightarrow\infty$ to the centered $N-$dimensional normal distribution
        $N(0,\Sigma)$ with the positive definite  covariance matrix ${\Sigma}$ whose inverse  
 	$S:=({s}_{j,k})_{j,k=1,...,N}:={\Sigma}^{-1}$ satisfies
 	\begin{align*}
 	 s_{j,k}=
 	\begin{cases}\sum_{ l: \> l\ne j}
 	  \frac{1}{(z_j-z_l)^2}+\frac{(N+a-1)(1-z_j^2)+2bz_j}{(1+z_j^2)^2} &\textit{ for } j=k\\
 	-\frac{1}{(z_j-z_k)^2}&\textit{ for } j\neq k.
 	\end{cases}
 	\end{align*}
\end{theorem}

In order to obtain further information about the  covariance matrices ${\Sigma}$ and their inverses,
we rewrite the CLT \ref{clt-hua-pickrell-alg} in trigonometric coordinates. For this  consider the diffeomorphism
$$T: \mathbb R\to \mathbb R, \quad (x_1,\ldots,x_N)\mapsto (\sinh(x_1),\ldots,\sinh(x_N)).$$
These trigonometric coordinates
are motivated by
corresponding results for Jacobi ensembles in \cite{HV} and by the  preceding section.

\begin{theorem}\label{clt-hua-pickrell-trig}
 	Let $a>0$, $b\in\mathbb R$, $\beta\ge1$, and $N\ge 2$ an integer.
 	Let $ X_{N,a,b,\beta}$ be random variables with the distributions $\rho_{N,a,b,\beta}$. 
 	Then the distributions of $\sqrt{\beta}( T^{-1}(X_{N,a,b,\beta})-T^{-1}(z))$
 	converge weakly for $\beta\rightarrow\infty$ to the centered $N-$dimensional normal distribution
        $N(0,\tilde\Sigma)$ with the  positive definite  covariance matrix ${\tilde\Sigma}$ whose inverse  
 	$\tilde S:=({\tilde s}_{j,k,j})_{j,k=1,...,N}:={\tilde \Sigma}^{-1}$ satisfies
 	\begin{align}\label{covmatrix-hua-trig}
 	 \tilde s_{j,k}&=   s_{j,k} \cdot \sqrt{z_j^2+1}\cdot \sqrt{z_k^2+1}\notag\\
 	&=\begin{cases}\sum_{ l: \> l\ne j}
 	  \frac{z_j^2+1}{(z_j-z_l)^2}+\frac{(N+a-1)(1-z_j^2)+2bz_j}{(1+z_j^2)} &\textit{ for } j=k\\
 	-\frac{ \sqrt{z_j^2+1}\cdot \sqrt{z_k^2+1}}{(z_j-z_k)^2}&\textit{ for } j\neq k.
 	\end{cases}
 	\end{align}
        The matrix $\tilde S$ has the simple  eigenvalues 
 \begin{equation}\label{eigenvalues-hua}	\lambda_k:=k(a+(k-1)/2)>0 \quad\quad (k=1,\ldots,N).	\end{equation}
 	Each  $\lambda_k$ has an eigenvector of the form
        \begin{equation}\label{eigenvectors-hua}	v_k:=\left(q_{k-1}(z_1)\sqrt{1+z_1^2},\ldots,q_{k-1}(z_N)\sqrt{1+z_N^2}\right)^T
          \quad\quad (k=1,\ldots,N)	\end{equation}
 	for polynomials $q_{k-1}$ of order $k-1$ which are orthonormal w.r.t. the discrete measure
 	\begin{align}\label{orth-measure-hua}
 	\mu_{N,\alpha,\beta}:=(1+z_1^2)\delta_{z_1}+\ldots+(1+z_N^2)\delta_{z_N}.
 	\end{align}
\end{theorem}

We point out that the equivalence of the CLTs themselves
in Theorems \ref{clt-hua-pickrell-alg} and \ref{clt-hua-pickrell-trig} with the relation between
the covariance matrices in \eqref{covmatrix-hua-trig} follow from the 
Delta-method for  CLTs of transformed random variables; see Section  3.1 of \cite{vV}.

We prove  Theorems \ref{clt-hua-pickrell-alg} and \ref{clt-hua-pickrell-trig} together. The proof will consist of two main steps:

We first show that in the setting of Theorem  \ref{clt-hua-pickrell-alg}, the Lebesgue
densities of $\sqrt{\beta}( X_{N,a,b,\beta}-z)$ tend locally uniformly on $\mathbb R^N$ to the density of $N(0,\Sigma)$ for $\beta\to\infty$.
We shall do this in algebraic coordinates, as the computation is easier in this case.
The computation of the limit of these densities decomposes into 2 parts, namely the non-constant part,
which involves the variable $x\in\mathbb R^N$, and the normalizing constants.
We shall restrict our attention to the first non-constant part below. The convergence of the constants then can be worked out by a
nasty computation using Stirling's formula similar to the $\beta$-ensembles in \cite{HV,V1}.  However, it is possible to use the Laplace
method (see \cite{B, W}) in order to obtain the correct convergence of the constants automatically.
This approach also will automatically
lead to weak
convergence of the involved probability measures. For convenience of the reader, we explain this approach  below in Theorems \ref{Laplace1}
and \ref{Laplace2}.
However, we point out that this approach via the Laplace method requires that the covariance matrix (which is a Hessian)
is positive definite.  From the very construction, positive semidefiniteness is clear, but positive definiteness is not obvious.
In our setting the positive definiteness follows from the second main part of the proof in which we compute
the eigenvalues and eigenvectors of $\tilde S$. This result then in particular implies that  $\tilde S$
and thus also $S$ is in fact positive definite. We here still mention that the  eigenvalues and eigenvectors of $\tilde S$
will be computed by using the corresponding result for Jacobi ensembles in Proposition 3.1 of \cite{HV}.
This method will also allow to compute the covariance matrices $\tilde \Sigma$ and $\Sigma$ explicitely by the methods in \cite{AHV}.

We next turn to the details of the proof.
We begin with the Laplace method in asymptotic analysis. In some standard form it is as follows; see
 Ch. IX of \cite{W} or Theorem 41 and Lemma 38 in \cite{B}:

\begin{theorem}\label{Laplace1}
Let  $\phi\in C^2(D)$ and
$\psi \in C_0(D)$ be functions  with $\phi,\psi \ge 0$ on some domain $D\subset \mathbb R^N$   with the following properties:
\begin{enumerate}
\item[\rm{(1)}] $\phi$ has  a unique global maximum at $z \in D$
such that for every neighborhood $V$ of $z$, 
$$\sup\{\phi (x) : \> x \in D\setminus V \} < \phi(z).$$
\item[\rm{(2)}]  $\phi$ has a negative definite Hessian matrix $H_\phi (z)$ in $z$.
\item[\rm{(3)}] The Lebesgue integrals $\int_D  \psi(x)\phi(x)^k\> dx$ ($k\in\mathbb N$) exist and converge  for $k\to\infty$.
\item[\rm{(4)}]  $\psi(z)>0$.
\end{enumerate}
Then,
$$\lim_{k\to\infty} \int_D  \psi(x)\phi(x)^k\> dx \cdot \Bigl(\frac{2\pi}{k}\Bigr)^{-N/2}\cdot
\frac{\det (- H_{\ln \phi} (z))^{1/2}    }{\psi(z)\phi(z)^k}=1.$$
\end{theorem}

Theorem \ref{Laplace1} implies the following general CLT; see Theorem 3.2 and Corollary 3.3 in \cite{He}:

\begin{theorem}\label{Laplace2}
  Let $D\subset\mathbb{R}^{N}$ be a domain, $\phi\in C^{3}(D)$, and $\psi\in C^0(D)$
  with values in $]0,\infty[$ and with the following properties:
 \begin{enumerate}
 \item[\rm{(1)}] $\phi$ has a global maximum at $z=(z_1,...,z_N)\in D$ such that for every neighborhood $V$ of $z$,
			$$\sup\{\phi(x):x\in D\setminus V\}<\phi(z);$$
\item[\rm{(2)}] The Hessian matrix $H_\phi(z)$ of  $\phi $ in $z$ is negative definite;
\item[\rm{(3)}] $c_\beta^{-1}:=\int_{D}\psi(x)\phi(x)^{\beta}dx<\infty$ for all $\beta>1$;
\item[\rm{(4)}] $\psi(z)> 0$.
  \end{enumerate}
Consider random variables $X_\beta$ on $D$ with the Lebesgue densities 	$f_\beta(x):=c_\beta\psi(x)\phi(x)^{\beta}$ for $\beta\ge1$. Then
the random variables
$\sqrt{\beta}(X_\beta-z)$ converge in distribution to $N(0,\Sigma)$ for $\beta\to\infty$ where the covariance matrix $\Sigma$ satisfies
$$\Sigma^{-1}=-H_{\ln\phi}(z)=-\frac{1}{\phi(z)}H_\phi(z).$$
Moreover,
the expectations $E(X_\beta)$ exist for $\beta\ge1$, and $\sqrt{\beta}(X_\beta-E(X_\beta))$ also converges in distribution to $N(0,\Sigma)$. 
	\end{theorem}

 We now apply Theorem \ref{Laplace2} to Hua-Pickrell ensembles.

 \begin{proof}[First main step of the proof of Theorems \ref{clt-hua-pickrell-alg} and \ref{clt-hua-pickrell-trig}]
   Fix parameters $a>0$, $b\in\mathbb R$ and choose the domain $D$ in Theorem \ref{Laplace2} as the interior of $C_N^A$.
   Let
   \begin{equation}\label{clt-proof1}
     \psi(x):=  \prod_{j=1}^N\bigl(1+ x_j^2\bigr)^{-1}
   \end{equation}
   and
   \begin{equation}\label{clt-proof2}
 \phi(x):= \exp\Bigl(\sum_{j=1}^N  \frac{1}{2}(1-N-a)\ln(1+x_j^2)
	+b\sum_{j=1}^N\arctan(x_j)+ \sum_{1\le j<k\le N}  \ln|x_j-x_k|\Bigr).
   \end{equation}
   Then the conditions (3) and (4) of Theorem \ref{Laplace2} hold obviously. Moreover, Lemma \ref{lem_pseudo_Jac_char}
   ensures that condition (1) holds and that
   the Hessian matrix $H_\phi(z)$  is negative semidefinite. As discussed above, we
   check below that $H_\phi(z)$ is in fact negative definite.
Furthermore,  for $l,k=1,\ldots, N$ with $k\ne l$, the function $g(x):=\ln\phi(x)$ satisfies
        \begin{equation}\label{clt-proof3}
\frac{d}{dx_l} g(z)= (1-N-a)\frac{z_l}{1+z_l^2} +\frac{b}{1+z_l^2}+ \sum_{j: \> j\ne l} \frac{1}{z_l-z_j},
 \end{equation}
        \begin{equation}\label{clt-proof4}
\frac{d^2}{dx_l \> dx_k} g(z)=\frac{1}{(z_l-z_k)^2},
        \end{equation}
        and
    \begin{equation}\label{clt-proof5}       
\frac{d^2}{dx_l^2} g(z)= \frac{ (1-N-a)(1-z_l^2) -2bz_l}{(1+z_l^2)^2} -\sum_{j: \> j\ne l}\frac{1}{(z_l-z_j)^2}.
          \end{equation}
  Therefore, up to the   negative definiteness of  $H_\phi(z)$,  Theorem \ref{clt-hua-pickrell-alg} follows from Theorem \ref{Laplace2}.
 \end{proof}

 We next check the eigenvalues and eigenvectors of the matrix $\tilde S$ in Theorem \ref{clt-hua-pickrell-trig}.


 \begin{proof}[Second main step of the proof of Theorem \ref{clt-hua-pickrell-trig}]
   We compare all data concerning the eigenvalues and eigenvectors in Theorem \ref{clt-hua-pickrell-trig} for the Hua-Pickrell case
   with the corresponding results for Jacobi ensembles in trigonometric coordinates in \cite{HV}:

   For this consider parameters  $\alpha,\beta>-1$ and the 
   zeroes $(z_1^{Jac},\ldots,z_N^{Jac})$ of the Jacobi polynomials $P_N^{(\alpha,\beta)}$. By Lemma  2.2 of \cite{HV} (see also Section 6.7 of \cite{Sz}),
   these zeroes are characterized uniquely by
\begin{equation}\label{lemma-formel1-jac}
          \sum_{l: l\neq j}\frac{1}{z_j^{Jac}-z_l^{Jac}}+\frac{\alpha+1}{2}\frac{1}{z_j^{Jac}-1}+\frac{\beta+1}{2}\frac{1}{z_j^{Jac}+1}=0
          \quad\text{for}\quad j=1,...,N,	\end{equation}
Furthermore,  by Proposition 3.1 of \cite{HV},  the matrix	${S^{Jac}}=(\tilde{s}^{Jac}_{j,k})_{j,k=1,...,N}$ with
 	\begin{equation}\label{jac1}
 	\tilde s^{Jac}_{j,k}:=
 	\begin{cases}4\sum_{ l:\> l\ne j}
 	  \frac{1-(z_j^{Jac})^2}{(z_j^{Jac}-z_l^{Jac})^2}+2(\alpha+1)\frac{1+z_j^{Jac}}{1-z_j^{Jac}}
          +2(\beta+1)\frac{1-z_j^{Jac}}{1+z_j^{Jac}} &\textit{ for } j=k\\
 	\frac{-4\sqrt{(1-(z_j^{Jac})^2)(1-(z_k^{Jac})^2)}}{(z_j^{Jac}-z_k^{Jac})^2}&\textit{ for } j\neq k
 	\end{cases}
 		\end{equation}
 	has the eigenvalues
        \begin{equation}\label{jac2}	\lambda_k=2k(2N+\alpha+\beta+1-k)>0 \quad\quad (k=1,\ldots,N)
	\end{equation}
 	with associated eigenvectors of the form
        \begin{equation}\label{jac3} v_k:=\left(q^{Jac}_{k-1}(z^{Jac}_1)\sqrt{1-(z_1^{Jac})^2},\ldots,q^{Jac}_{k-1}(z^{Jac}_N)\sqrt{1-(z_N^{Jac})^2}\right)^T
	\end{equation}
 	for polynomials $q^{Jac}_{k-1}$ of order $k-1$ which are orthonormal w.r.t the discrete measure
 	\begin{equation}\label{jac4}
 	  \mu_{N,\alpha,\beta}:=(1-(z_1^{Jac})^2)\delta_{z_1^{Jac}}+\ldots+(1-(z_N^{Jac})^2)\delta_{z_N^{Jac}}.	\end{equation}
        
        If we take $\alpha:=-a-N+ib$, $\beta:=-a-N-ib$, and $z_j^{Jac}= i z_j$, then   \eqref{lemma-formel1-jac} is transformed into
        the identity in Lemma \ref{lem_pseudo_Jac_char}(2) and the matrix	$-S^{Jac}/4$ into $\tilde S$.
As the proof of  Proposition 3.1 of \cite{HV} just consists of linear algebra together with  \eqref{lemma-formel1-jac} 
and   without any further input, the computations in proof of  Proposition 3.1 of \cite{HV} imply readily
\eqref{eigenvalues-hua}-\eqref{orth-measure-hua} (where the minus sign has to be taken into account).
 This completes the proof of Theorems \ref{clt-hua-pickrell-alg} and \ref{clt-hua-pickrell-trig}.
 \end{proof}

 We point out that we are not able to determine the eigenvalues and eigenvectors for  the matrix $S$ in Theorem \ref{clt-hua-pickrell-alg}.
 However, the eigenvalues of $\tilde S$ and the relation in \eqref{covmatrix-hua-trig} between $S$ and  $\tilde S$ yields:

 \begin{corollary} The matrix $S$ of Theorem \ref{clt-hua-pickrell-alg} satisfies
$$\det S =  \frac{N! \prod_{j=0}^{N-1}(2a+j)^3}{2^{3N}\prod_{j=0}^{N-1}((a+j)^2+b^2)}. $$
\end{corollary}

 \begin{proof}
   Write the pseudo-Jacobi polynomials  as
   $P_N(x;-N-a,b)=l_N^{a,b}\prod_{j=1}^N (x-z_j),$
   where, by \eqref{def-pseudo-jac}, the leading coefficients $l_N^{a,b}$ is given by
   $$l_N^{a,b}= \frac{1}{N!\cdot 2^N}\prod_{j=1}^N (-2a-N+j).$$
   Hence,
   \begin{align}
     \prod_{j=1}^N (1+z_j^2)&= \prod_{j=1}^N (1+iz_j) \cdot \prod_{j=1}^N (1-iz_j) =
     \frac{1}{(l_N^{a,b})^2} P_N(-i;-N-a,b)P_N(i;-N-a,b) \notag\\
     &=  \frac{1}{(l_N^{a,b})^2} P_N^{(-a-N+ib, -a-N-ib)}(1)P_N^{(-a-N+ib, -a-N-ib)}(-1) \notag
     \end{align}
   and thus, by Section 4.1 of \cite{Sz},
\begin{equation}
 \prod_{j=1}^N (1+z_j^2)= \frac{1}{(l_N^{a,b})^2} {-a+ib\choose N}{-a-ib\choose N}
=\frac{2^{2N} \prod_{j=0}^{N-1}((a+j)^2+b^2)}{ \prod_{j=0}^{N-1}(2a+j)^2}. \notag
     \end{equation}
 This, the eigenvalues of $\tilde S$, and  the relation in \eqref{covmatrix-hua-trig} yield the claim.
 \end{proof}

 \begin{remarks}
\begin{enumerate}
 \item[\rm{(1)}] The covariance matrices $\Sigma$ and $\tilde\Sigma$ in the CLTs  \ref{clt-hua-pickrell-alg} and
\ref{clt-hua-pickrell-trig} can be computed from the eigenvectors and eigenvalues in Theorem \ref{clt-hua-pickrell-trig}
similar to the corresponding results for $\beta$-Hermite, Laguerre, and Jacobi ensembles in \cite{AHV}; see also  \cite{DE, GK} for different approaches
in the  $\beta$-Hermite and  Laguerre cases.
\item[\rm{(2)}] We in particular point out in the case of $\beta$-Hermite ensembles and Dyson Brownian motions,
  the models for dimensions $N+1$ and $N$ are related by well-known interlacing identities, which were used in \cite{GK}
   to derive an interlacing extension of the
   freezing CLTs mentioned above in this case. This approach leads directly to the  covariance matrices $\Sigma$, but not to their inverses
   and the associated spectra as above. Such interlacing couplings for different $N$ are available also for further models
   like for $\beta$-Jacobi and  Hua-Pickrell ensembles; see \cite{AN, AOW} and references there.
   We therefore expect that the freezing CLTs above may be also extended to interlacing extension similar to \cite{GK}.
\end{enumerate}
\end{remarks}

\section{Limits of the empirical measures of Hua-Pickrell diffusions}\label{subsec_HP_natur_lim}

In this section, we analyze the almost sure weak limits of empirical measures of Hua-Pickrell diffusions for $N\to\infty$.
For this we choose the parameters $\beta, a,b$ in \eqref{eq_SDE_Hua-Pickrell-norm} in dependence on $N$. More precisely,
for all $N$ we  consider the solutions $(X_t^N)_{t\geq0}$
of the SDEs
\begin{equation}\label{eq_SDE_N_rescale}
	dX_{j,t}^N=2\sqrt{\frac{1+\left(X^N_{j,t}\right)^2}{\beta N}}\,dB_{j,t}+\frac{2}{N}\left[b_N-a_NX_{j,t}^N+\sum_{k\colon k\neq j}\frac{X^N_{j,t}X^N_{k,t}+1}{X^N_{j,t}-X^N_{k,t}}\right]dt
\,,\;\; (j=1,\ldots,N)
\end{equation}
on $C_N^A$ with $\beta\in[1,\infty],a_N,b_N\in\mathbb R$  and initial conditions $X_0^N=\hat{x}_0^N\in C_N^A$.
Please notice that, in comparison to \eqref{eq_SDE_Hua-Pickrell-norm}, we here use the time scaling $t\mapsto t/N$.
The empirical measures associated to $(X_t^N)_{t\geq0}$ are now defined by
$$
	\mu_{N,t}:=\frac{1}{N}\sum_{j=1}^N\delta_{X_{j,t}^N} \quad (t\ge0).
$$
        The following theorem yields that the empirical measures $\mu_{N,t}$ converge weakly a.s.\ for $N\to\infty$,
        whenever this holds for $t=0$:

\begin{theorem}\label{thm_exist_limit_emp_meas}
	Let $\beta\in[1,\infty]$.
	Assume that the limits $\lim_{N\to\infty}a_N/N=:\hat{a}$ and $\lim_{N\to\infty}b_N/N=:\hat{b}$ exist.
	Assume that there exists a probability measure  $\mu\in M^1(\mathbb{R})$ with
    \begin{equation}\label{eq_Carleman}
    	\left\lvert\int_{\mathbb{R}}x^n\,d\mu(x)\right\rvert\leq(\gamma n)^n\;\;\text{for all}\;\;n\in\mathbb{N}\;\;\text{and some}\;\;\gamma>0
	\end{equation}
such that all moments of the empirical measures
	$$
		\mu_{N,0}=\frac{1}{N}\sum_{j=1}^N\delta_{\hat{x}_{j,0}^N}
	$$
	converge to those of $\mu$ for $N\to\infty$. 
	Then there exist probability measures $(\mu_t)_{t\geq0}\subset M^1(\mathbb{R})$ with $\mu_0=\mu$ such that
	$$
		\mu_{N,t}\xrightarrow{N\to\infty}\mu_t\;\;\text{weakly for all}\;\;t\geq0\;\;\text{a.s.}
	$$
	        The moments $m_n(t):=\int_{\mathbb{R}}x^n\,d\mu_t(x)$  are finite for  $n\in\mathbb{N}$ and $t\geq0$
                and satisfy the recursive ODEs
	\begin{equation}\label{eq_free_HP_mom_recur}
			\frac{d}{dt}m_n(t)
			=n\left[-2(\hat{a}+1)m_n+2\hat{b}m_{n-1}(t)+\sum_{k=0}^{n}m_{k}(t)m_{n-k}(t)
			+\sum_{k=0}^{n-2}m_k(t)m_{n-2-k}(t)\right]\,.
	\end{equation}
	The corresponding Cauchy transforms $G(t,z):=\int_{\mathbb{R}}\frac{1}{z-x}\,d\mu_t(x)$, $\Re(z)\neq0$, satisfy the PDE
	\begin{equation}\label{eq_HP_Cauchy_PDE}
		\partial_tG(t,z)
		=-\partial_z\left[\left(-2(\hat{a}+1)z+2\hat{b}\right)G(t,z)+\left(z^2+1\right)G^2(t,z)\right]\,.
	\end{equation}
	Moreover, if $\mu$ is compactly supported, then for each $T>0$,
        $$\bigcup_{t\in[0,T]} \operatorname{supp}(\mu_t)\subset\mathbb R$$ 
        is bounded.
\end{theorem}

\begin{proof}
	The proof uses the same techniques as used in \cite{AVW} and Section 3 of \cite{AV}.
	For the convenience of the reader we present the main steps of this approach:\\
	1st step: We  show the convergence of the $n$-th empirical moment
	\begin{equation*}
		S_{N,n,t}:=\frac{1}{N}\sum_{j=1}^N\left(X_{j,t}^N\right)^n=\int_{\mathbb{R}}x^n\,d\mu_{N,t}(x)\,,\;\;n\in\mathbb{N}_0\,,
	\end{equation*}
	where we  use the convention $S_{N,-1,t}:=0$.
	 It\^{o}'s formula and 
	$$
		2\sum_{j,l=1,\,j\neq l}^N\frac{x_j^{m+1}}{x_j-x_l}=\sum_{k=0}^m\sum_{j=1}^Nx_j\sum_{l=1}^Nx_l^{m-k}-(m+1)\sum_{j=1}^Nx_j\,,\;\;m\in\mathbb{N}_0\,,\;\;x_1<\dots<x_N\,,
	$$
	imply that the $n$-the empirical moment satisfies
	\begin{equation}\label{eq_pf_moment_cvg_recur_finite}
	\begin{split}
		dS_{N,n,t}=&dM_{N,n,t}+n\left[\left(-\frac{2}{N}(a_N-1)-\frac{n+1}{N}+2\frac{n-1}{\beta N}\right)S_{N,n,t}+\frac{2b_N}{N}S_{N,n-1,t}\right.\\
		&\left.+\frac{1}{N}\left(\frac{2}{\beta}-1\right)(n-1)S_{N,n-2,t}+\sum_{k=1}^{n-1}S_{N,k,t}S_{N,n-k,t}+\sum_{k=0}^{n-2}S_{N,k,t}S_{N,n-2-k,t}\right]dt\,,
	\end{split}
	\end{equation}
	where
	\begin{equation*}
		M_{N,n,t}=\frac{2n}{N}\sum_{j=1}^N\int_0^t\left(X_{j,s}^N\right)^{n-1}\sqrt{\frac{1+\left(X_{j,s}^N\right)^2}{\beta N}}\,dB_{j,s}\,.
	\end{equation*}
	As in the proof of Theorem 3.1 in \cite{AV},  a stopping time argument, the elementary estimate
	$$
	\left\lvert\frac{1}{N}\sum_{j=1}^Nx_j^k\right\rvert\leq\left(\frac{1}{N}\sum_{j=1}^N\lvert x_j\rvert^m\right)^{k/m}\leq1+\sum_{j=1}^N\lvert x_j\rvert^m\,,\;\;k,m\in\mathbb{N}\,,\;\;k\leq m\,,
	$$
	and Gronwall's Lemma then show that $\left(M_{N,n,t}\right)_{t\geq0}$ is an $L^2$-martingale, and that 
	\begin{equation}\label{eq_pf_moment_cvg_Gronwall}
		E\left(S_{N,2n,t}\right)\leq\left(S_{N,2n,0}+C(N,n,\beta)\right)\exp(C(N,n,\beta)t)
	\end{equation}
	for some $C(N,n,\beta)>0$ satisfying $\limsup_{N\to\infty}C(N,n,\beta)<\infty$.\\
	The next main step is to show that  $\left(M_{N,n,t}\right)_{t\geq0}$ converges to $0$, i.e., for all $T>0$,
	\begin{equation}\label{eq_pf_moment_cvgc_mart}
		\lim_{N\to\infty}\sup_{t\in[0,T]}\lvert M_{N,n,t}\rvert=0\;\;\text{for all}\;\;n\in\mathbb{N}\;\;\text{a.s.}
	\end{equation}
	To see this, we first note that by It\^{o}'s isometry the corresponding quadratic variation is 
	\begin{equation}
		\left[M_{N,n}\right]_t=4\frac{n^2}{\beta N^2}\int_0^t(S_{N,2n,s}+S_{N,2(n-1),s})\,ds\,.
	\end{equation}
	Using the Markov and Burkholder-Davis-Gundy inequality, we  obtain
        that there is a constant $c$  independent of $N$ such that for all $\epsilon>0$
	\begin{equation}
		P\left(\sup_{t\in[0,T]}\lvert M_{N,n,t}\rvert\geq\epsilon\right)
		\leq\frac{4cn^2}{\beta N^2\epsilon^2}\int_0^T\left(E(S_{N,2n,s})+E(S_{N,2(n-1),s})\right)\,ds.
	\end{equation}
	Combining this with \eqref{eq_pf_moment_cvg_Gronwall} and our assumption on $\hat{x}_0^N$, we see that
        $$P\left(\sup_{t\in[0,T]}\lvert M_{N,n,t}\rvert\geq\epsilon\right)=\mathcal{O}(N^{-2})$$ for $N\to\infty$.
	By the Borel-Cantelli Lemma we hence know that \eqref{eq_pf_moment_cvgc_mart} holds.
	Setting
	$$
		\lambda_n:=\lambda_n(N,\beta):=n\left(-\frac{2}{N}(a_N-1)-\frac{n+1}{N}+2\frac{n-1}{\beta N}\right)
	$$
	we deduce from \eqref{eq_pf_moment_cvg_recur_finite} that
	\begin{equation*}
	\begin{split}
		S_{N,n,t}=&e^{\lambda_nt}\left(\vphantom{\sum_k^n}S_{N,n,0}+M_{N,n,t}\right.\\
		&\left.+\int_0^te^{-\lambda_ns}\left(\lambda_nM_{N,n,s}+n\left[\vphantom{\sum_k^n}\frac{2b_N}{N}S_{N,n-1,s}+\frac{1}{N}\left(\frac{2}{\beta}-1\right)(n-1)S_{N,n-2,s}\right.\right.\right.\\
		&\left.\left.\left.+\sum_{k=1}^{n-1}S_{N,k,s}S_{N,n-k,s}+\sum_{k=0}^{n-2}S_{N,k,s}S_{N,n-2-k,s}\right]\right)\,ds\right)\,.
	\end{split}
	\end{equation*}
	By induction, it follows that there are continuous functions $\hat{m}_n\colon[0,\infty)\to\mathbb{R}$ such that
	$$
		\lim_{N\to\infty}\sup_{t\in[0,T]}\lvert S_{N,n,t}-\hat{m}_n(t)\rvert=0\;\;\text{for all}\;\;n\in\mathbb{N}\;\;\text{a.s.}
	$$
	Moreover, these functions satisfy
	\begin{equation}\label{pf_eq_mom_cvgc_mom_expl}
		\begin{split}
			\hat{m}_n(t)=&e^{-2n\hat{a}t}\left(\vphantom{\sum_k^n}\hat{m}_n(0)\right.\\
			&\left.+n\int_0^te^{2n\hat{a}s}\left(2\hat{b}\hat{m}_{n-1}(s)+\sum_{k=1}^{n-1}\hat{m}_k(s)\hat{m}_{n-k}(s)+\sum_{k=0}^{n-2}\hat{m}_{k}(s)\hat{m}_{n-2-k}(s)\right)\,ds\right)\,.
		\end{split}
	\end{equation}
	In particular, $\hat{m}_n$ is a solution of the ODE \eqref{eq_free_HP_mom_recur}.\\
	2nd step: We show that a bound similar to \eqref{eq_Carleman} holds for $\hat{m}_n(t)$ for $t\geq0$.
	First, note that by our assumptions  $\lvert\hat{m}_n(0)\rvert\leq(\gamma n)^n$ for all $n\in\mathbb{N}$.
	For $t>0$ we now show by induction on $n$ that
	\begin{equation}\label{eq_pf_mom_cvgc_carleman}
	  \lvert\hat{m}_n(t)\rvert\leq\left(nc_1e^{c_2t}\right)^n \quad \text{with}\quad c_1:=\max(1,2\gamma,4(\hat{b}+5)),
         c_2:=\max(0,c_1-2\hat{a}) .
	\end{equation}
	In fact,  for $n=1,2$ this follows immediately from \eqref{pf_eq_mom_cvgc_mom_expl}.
	In the induction step $n\geq3$ we apply the inequality (cf. proof of Thm 3.1 in \cite{AV})
	$$
		\sum_{k=1}^{n-1}k^k(n-k)^{n-k}\leq4(n-1)^{n-1}
	$$
	and obtain, as claimed, 
	\begin{equation}
	\begin{split}
		\lvert\hat{m}_n(t)\rvert
		\leq&e^{-2\hat{a}nt}\left((\gamma n)^n+n\int_0^te^{2\hat{a}ns}\left(2\lvert\hat{b}\rvert\lvert\hat{m}_{n-1}(s)\rvert+2\lvert\hat{m}_{n-2}(s)\rvert+\sum_{k=1}^{n-1}\lvert\hat{m}_k(s)\rvert\lvert\hat{m}_{n-k}(s)\rvert\right.\right.\\
		&\left.\left.+\sum_{k=1}^{n-3}\lvert\hat{m}_k(s)\rvert\lvert\hat{m}_{n-2-k}(s)\rvert\right)ds\right)\\
		\leq&e^{-2\hat{a}nt}\left((\gamma n)^n+n\int_0^te^{n(c_2+2\hat{a})s}c_1^n\left(2\lvert\hat{b}\rvert(n-1)^{n-1}+2(n-2)^{n-2}+\sum_{k=1}^{n-1}k^k(n-k)^{n-k}\right.\right.\\
		&\left.\left.+\sum_{k=1}^{n-3}k^k(n-2-k)^{n-2-k}\right)ds\right)\\
		\leq&e^{-2\hat{a}nt}\left((\gamma n)^n+e^{n(c_2+2\hat{a})t}\frac{1}{c_2+2\hat{a}}c_1^n\left(2\lvert\hat{b}\rvert(n-1)^{n-1}+2(n-2)^{n-2}+4(n-1)^{n-1}\right.\right.\\
		&\left.\vphantom{\frac{1}{\hat{c}_2}}\left.\vphantom{2\lvert\hat{b}\rvert(n-1)^{n-1}+2(n-2)^{n-2}+4(n-1)^{n-1}}+4(n-3)^{n-3}\right)\right)\\
		\leq&(nc_1e^{c_2t})^n\,.
	\end{split}
	\end{equation}
	3rd step: To complete the proof of the convergence of $\mu_{N,t}$ for $N\to\infty$, we note that by \eqref{eq_pf_mom_cvgc_carleman} for each $t\geq0$ the sequence $(\hat{m}_n(t))_{n\in\mathbb{N}}$ satisfies
	$$
		\sum_{n=1}^{\infty}\left(\hat{m}_{2n}(t)\right)^{-2n}=\infty\,.
	$$
	        By Carleman's condition (cf. \cite{Ak} p.85) and the moment convergence theorem,
                we thus know that there exist probability measures $\mu_t$ with  $\mu_{N,t}\xrightarrow{N\to\infty}\mu_t$ a.s for each $t\geq0$, where
                the moments of the $\mu_t$ are finite and given by $\hat{m}_n(t)$ for $n\in\mathbb{N}$.\\
	        4th step: To show that the Cauchy-transforms $G(t,z)$ of the measures $\mu_t$, $t\geq0$ are solution to \eqref{eq_HP_Cauchy_PDE} we
                use the same approximation argument as in \cite{AVW}.
	        To be more precise, one first derives a PDE for the Cauchy-transforms $G^N(t,z):=\int_{\mathbb{R}}\frac{1}{z-x}\,d\mu_{N,t}(x)$ for
                $\beta=\infty$ using \eqref{eq_pf_moment_cvg_recur_finite} and the relation
	$$
		G^N(t,z)=\sum_{k\geq0}z^{-(k+1)}S_{N,k,t}\,.
	$$
	This PDE is 
	\begin{equation*}
	\begin{split}
		&\partial_tG^N(t,z)\\
		=&-\partial_z\left(\frac{1}{N}\partial_z\left((1+z^2)G^N(t,z)\right)+2\left(-\left(\frac{a^N}{N}+1\right)z+\frac{b_N}{N}\right)G^N(t,z)+(1+z^2)\left(G^N(t,z)\right)^2\right)\,.
	\end{split}
	\end{equation*}
It is then easy to see that the coefficients of this PDE converge to the corresponding coefficients in \eqref{eq_HP_Cauchy_PDE}.
	The rigorous argument which shows that the Cauchy-transform of the weak limit of $\mu_{N,t}$ actually is a solution of \eqref{eq_HP_Cauchy_PDE} is the same as in the proof of Proposition 2.9 in \cite{VW1}.\\
	Final step: The statement about uniform compact support follows from Theorem \ref{prop_HP_moment_recur} below.
\end{proof}

We next study  some properties of the measures $(\mu_t)_{t\geq0}$ in Theorem \ref{thm_exist_limit_emp_meas}, which depend on  $\hat{a},\hat{b}\in\mathbb{R}$ only, and not  on $\beta\in[1,\infty]$.
We start with the behavior of these measures for $t\to\infty$.
For this, we need several well-known measures.

For $a,b\in\mathbb{R}$ with $a>0$ we define the {\it equilibrium Hua-Pickrell measure with parameters $a,b$}
as the probability measure on $\mathbb{R}$ with density
\begin{equation*}
	\mu_{\operatorname{HP},a,b}(dx):=
	\begin{cases}
		\frac{a\sqrt{(x-x_-)(x_+-x)}}{\pi(1+x^2)}dx\,,\quad &x\in(x_-,x_+)\,,\\
		0\,,\quad&\text{otherwise}\,
	\end{cases}
\end{equation*}
with
$$
x_{\pm}:=\frac{1}{a^2}\left(-b(a+1)
\pm\sqrt{(2a+1)(a^2+b^2)}\right)\,.
$$
It is known that the  $\mu_{\operatorname{HP},a,b}$ are the weak  limits for $N\to\infty$  of the empirical measures of the zeros of pseudo-Jacobi polynomials
(\cite{JT} Theorem 2.2) and of empirical measures of Hua-Pickrell (or Cauchy) ensembles by Section 3.2.2 of \cite{FR}.
This fits to the following result:

\begin{lemma}\label{lem_HP_limit_long_time}
  Let $\hat{a},\hat{b}$ and $(\mu_t)_{t\geq0}$ be as in Theorem \ref{thm_exist_limit_emp_meas} with $\hat{a}>0$. Then
	$$
		\lim_{t\to\infty}\mu_{t}=\mu_{\operatorname{HP},\hat{a},\hat{b}}\,
	$$
weakly.	Moreover, if $\mu_{0}=\mu_{\operatorname{HP},\hat{a},\hat{b}}$, then $\mu_{t}=\mu_{\operatorname{HP},\hat{a},\hat{b}}$ for all $t\ge0$.
\end{lemma}

\begin{proof}
	If $\mu_0=\mu_{\operatorname{HP},\hat{a},\hat{b}}$, then by the proof of Theorem 2.2 in \cite{JT} its Cauchy transform $G_{\mu_{\operatorname{HP},\hat{a},\hat{b}}}:=\int_{\mathbb{R}}\frac{1}{z-x}\,d\mu_{\operatorname{HP},\hat{a},\hat{b}}(x)$ satisfies
	$$
	(z^2+1)G_{\mu_{\operatorname{HP},\hat{a},\hat{b}}}^2(z)-2((\hat{a}+1)z-\hat{b})G_{\mu_{\operatorname{HP},\hat{a},\hat{b}}}(z)-(2\hat{a} + 1) = 0\,.
	$$
	Hence, $(t,z)\mapsto G_{\mu}(z)$ is a solution to the PDE \eqref{eq_HP_Cauchy_PDE}.
	On the other hand, by Theorem \ref{thm_exist_limit_emp_meas}, $G(t,z):=\int_{\mathbb{R}}\frac{1}{z-x}\,d\mu_t(x)$ also is solution of this PDE,
        and $(\mu_t)_{t\geq0}$ is locally uniformly compactly supported.
	In particular, the relation $\int_{\mathbb{R}}\frac{1}{z-x}\,d\nu(x)=\sum_{k=0}^{\infty}z^{-(k+1)}\int_{\mathbb{R}}x^k\,d\nu(x)$ between Cauchy-transform and moments of compactly supported measures as well as the uniqueness to solutions of the ODEs \eqref{eq_free_HP_mom_recur} imply that $G(t,\cdot)=G_{\mu_{\operatorname{HP},\hat{a},\hat{b}}}(\cdot)$ for all $t\geq0$.\\
	It remains to discuss the weak convergence of $\mu_{t}$ for $t\to\infty$ for arbitrary starting conditions.
	Denote the moments of $\mu_t$ by $m_n(t):=\int_{\mathbb{R}}x^n\,d\mu_t(x)$.
	Then by \eqref{eq_free_HP_mom_recur},
	\begin{equation*}
		\begin{split}
			m_n(t)
			=&e^{-2n\hat{a}t}\left(m_n(0)+\int_0^te^{2n\hat{a}s}\left(2\hat{b}m_{n-1}(s)+\sum_{k=1}^{n-1}m_{k}(s)m_{n-k}(s)+\sum_{k=0}^{n-2}m_{k}(s)m_{n-k}(s)\right)ds\right)\,.
		\end{split}
	\end{equation*}
	Induction on $n$ now shows that $\lim_{t\to\infty}m_n(t)$ exists and that this limit does not depend on $m_n(0)$.
	Hence, the limit is the $n$-th moment of $\mu_{\operatorname{HP},\hat{a},\hat{b}}$ which is the constant solution by the first part of the proof.
	The claim follows from the moment convergence theorem.
\end{proof}

For certain parameters $\hat{a},\hat{b}$ it is possible to identify the $\mu_t$
in Theorem \ref{thm_exist_limit_emp_meas} for general starting conditions via an explicit transform of the {\it free positive multiplicative Brownian motion}
studied in \cite{Z} and references therein.
To state this result, we denote by $\mu_{\operatorname{sc},R}\in M^1(\mathbb R)$ the semicircle distributions with radius $R>0$, which have
 densities
$$\frac{2}{\pi R^2}\sqrt{R^2-x^2} {\bf 1}_{[-R,R]}.$$
 Furthermore, for $a<b$ let $\operatorname{Unif}_{[a,b]}$ the uniform distribution on $[a,b]$ with density $\frac{1}{b-a}{\bf 1}_{[a,b]}$.

\begin{theorem}
  Let $(\mu_t)_{t\geq0}$ be the  probability measures from Theorem \ref{thm_exist_limit_emp_meas} such that
 $\hat{a}=-1$ and $\hat{b}=0$ holds,  and such that the starting distribution $\mu$ is even.
  Let $\nu\in M^1((0,\infty))$ be the unique  measure
which is invariant under the pushforward by $x\mapsto x^{-1}$, and which satisifies $f(\nu)=\mu$ with
	$$
		f(x):=\frac{x^{1/2}-x^{-1/2}}{2}\,.
	$$
	Then
	$$
	\mu_t
	=f\left(\nu\boxtimes\exp\left(\mu_{\operatorname{sc},4\sqrt{2t}}\boxplus\operatorname{Unif}_{[-4t,4t]}\right)\right)\;\;\text{for all}\;\;t\geq0\,,
	$$
	where $f(\cdot),\exp(\cdot)$ denote the pushforward measures under  $f,\exp$,  and where $\boxtimes,\boxplus$ are
        the usual free additive and multiplicative convolutions (see \cite{MS,BV}).
\end{theorem}

\begin{proof}
%
	Let $\nu,\mu$ be probability measures as in the assumption.
	Let $\hat{y}_0^N=(\hat{y}_{1,0}^N,\dots,\hat{y}_{N,0}^N)\in C_N^{A}\cap(0,\infty)^N$, $N\in\mathbb{N}$, be a sequence of vectors such that
	\begin{enumerate}
		\item[(i)] $y_{j,0}^N=(y_{N+1-j,0}^N)^{-1}$, $j\in\{1,\dots,N\}$, and
		\item[(ii)] the empirical measures $\frac{1}{N}\sum_{j=1}^N\delta_{\hat{y}_{j,0}}^N$ converge for $N\to\infty$ in moments to $\nu$.
	\end{enumerate}
	Let $\left(y_t^N=(y_{1,t}^N,\dots,y_{N,t}^N)\right)_{t\geq0}$ be solutions of \eqref{eq_froz_mult_BM_ODE}
with start in $y_0^N=\hat{y}_0^N$ for 
        $N\in\mathbb{N}$.
	Then by Theorems 3.1 and 3.2 in \cite{AV},
	\begin{equation}\label{eq_pf_Thm_2_emp_meas_limit_1}
		\frac{1}{N}\sum_{j=1}^N\delta_{y_{j,t/N}^N}\xrightarrow{N\to\infty}\nu\boxtimes\exp\left(\mu_{\operatorname{sc},2\sqrt{t/2}}\boxplus\operatorname{Unif}_{[-t/2,t/2]}\right)
	\end{equation}
	where we used that the pushforward under the linear map $D_{a}(x):=ax$ satisfies
	$$
		D_2(\mu_{\operatorname{sc},2\sqrt{t/2}}\boxplus\operatorname{Unif}_{[-t/2,t/2]})
		=D_2(\mu_{\operatorname{sc},2\sqrt{t/2}})\boxplus D_2(\operatorname{Unif}_{[-t/2,t/2]})
		=\mu_{\operatorname{sc},2\sqrt{2t}}\boxplus\operatorname{Unif}_{[-t,t]}\,.
	$$
	On the other hand, we know by Lemma \ref{symmetry-heckman-a} that
	\begin{equation}\label{eq_pf_Thm_2_emp_meas_limit_2}
		x_t^N:=(f(y_{1,4t}^N),\dots,f(y_{N,4t}^N))\,,\,\,t\geq0\,,
	\end{equation}
	is a frozen Hua-Pickrell diffusion with parameters $a_N=-N+\frac{1}{2}$, $b_N=0$ and start in $\hat{x}_0^N:=(f(\hat{y}_{1,0}^N),\dots,f(\hat{y}_{N,0}^N))$.
	In particular,  by Theorem 3.1,  the corresponding empirical measures satisfy
	$$
		\frac{1}{N}\sum_{j=1}^N\delta_{x_{j,t/N}^N}\xrightarrow{N\to\infty}\mu_t\;\;\text{for all}\;\;t\geq0\,.
	$$
	This, \eqref{eq_pf_Thm_2_emp_meas_limit_1}, and \eqref{eq_pf_Thm_2_emp_meas_limit_2} now yield the claim.
\end{proof}

We next describe the connection of the measures $(\mu_t)_{t\geq0}$  in Theorem \ref{thm_exist_limit_emp_meas} with free It\^{o} processes.
These processes are $N=\infty$-analogues of It\^{o} processes with values in the space of square matrices of size $N$.
For an in depth introduction to these processes as well as the corresponding It\^{o} calculus see \cite{N}.

\begin{definition}
	Let $(\mathcal{A}, (\mathcal{A}_t)_{t\geq0},\tau)$ be a filtered $W^{*}$-probability space.
	Let $(c_t)_{t\geq0}$ be a circular Brownian motion on $(\mathcal{A}, (\mathcal{A}_t)_{t\geq0},\tau)$ (cf. \cite{N}).
	Let $\hat{y}_0$ be some self-adjoint bounded operator freely independent of $(c_t)_{t\geq0}$.
	Denote the identity operator in $\mathcal{A}$ by $\mathbf{1}$.
	We call a solution to the free stochastic differential equation
	\begin{equation}\label{eq_fSDE_HP}
		dy_t
		=dc_t\,\sqrt{\mathbf{1}+y^2_t}+\sqrt{\mathbf{1}+y^2_t}\,dc^*_t+\left[-(1+2a)y_t+(2b+\tau(y_t))\mathbf{1}\right]dt\,,\;y_0=\hat{y}_0\,,
	\end{equation}
	{\it free Hua-Pickrell process with parameter $a,b$ driven by $(c_t)_{t\geq0}$}; see Section 7 of \cite{As2} for the corresponding matrix process.
\end{definition}

\begin{lemma}
	Let $a,b\in\mathbb{R}$, let $(c_t)_{t\geq0}$ be a circular Brownian motion on some filtered $W^{*}$-probability space $(\mathcal{A}, (\mathcal{A}_t)_{t\geq0},\tau)$, and let $\hat{y}_0$ be a self-adjoint bounded operator free from $(c_t)_{t\geq0}$.
	Then the free stochastic differential equation \eqref{eq_fSDE_HP} has a unique solution.
\end{lemma}

\begin{proof}
	This essentially follows from Proposition A.1 in \cite{CD}.
	Here, we use that $x\mapsto\sqrt{1+x^2}$ and affine linear maps are in $C^2(\mathbb{R})$ and hence locally operator Lipschitz (see \cite{BS2} Section 2.3).
	Furthermore, for self-adjoint $x\in\mathcal{A}$ we have $\mathbf{1}+x^2\leq(\mathbf{1}+\lvert x\rvert)^2$.
	Hence, by Theorem 3.1.12 of \cite{St},  $\sqrt{\mathbf{1}+x^2}\leq\mathbf{1}+\lvert x\rvert$.
	This in turn implies the growth bound $\lvert\lvert\sqrt{\mathbf{1}+x^2}\rvert\rvert^2\leq2(1+\lvert\lvert x\rvert\rvert^2)$.
	Finally, note that the map $\mathcal{A}\to\mathcal{A}$, $A\to\tau(A)\mathbf{1}$ does not match the assumption of the cited Proposition in the strict sense.
	However, it is clear from the corresponding proof that it suffices to show
	$$
		\lvert\lvert\tau(A)\mathbf{1}-\tau(B)\mathbf{1}\rvert\rvert\leq\lvert\lvert A-B\rvert\rvert
		\;\;\text{and}\;\;
		\lvert\lvert\tau(A)\mathbf{1}\rvert\rvert^2\leq1+\lvert\lvert A\rvert\rvert^2\;\;\text{for all}\;\;A,B\in\mathcal{A}\,.
	$$
	These inequalities directly follow from the  general von Neumann algebra estimate for tracial states $\lvert\tau(A)\rvert\leq\lvert\lvert A\rvert\rvert$ for  $A\in\mathcal{A}$; see e.g. Lemma 2.1.3 of \cite{dS}.
\end{proof}

\begin{proposition}\label{prop_HP_moment_recur}
	Let $\hat{a},\hat{b}\in\mathbb{R}$.
	Let $(y_t)_{t\geq0}$ be a free Hua-Pickrell process with parameters $\hat{a},\hat{b}\in\mathbb{R}$ on some filtered $W^{*}$-probability space $(\mathcal{A}, (\mathcal{A}_t)_{t\geq0},\tau)$.
	Then the moments $m_n(t):=\tau(y^n(t))$, $n\in\mathbb{N}_0$ satisfy the ODEs \eqref{eq_free_HP_mom_recur}.\\
	Let $(\mu_t)_{t\geq0}$ be the family of probability measures in Theorem \ref{thm_exist_limit_emp_meas} for the parameters $\hat{a},\hat{b}$ and in the case that $\mu_0$ equals the spectral distribution of $y_0$.
	Then the spectral distribution of $y_t$ coincides with the measure $\mu_t$ for all $t\geq0$.
\end{proposition}

\begin{proof}
	By the traced It\^{o} formula (\cite{N} Theorem 3.5.3), we have
	\begin{align*}
		d\tau(y^n_t)
		=&n\left[\vphantom{\sum_k^n}-(1+2\hat{a})\tau(y_t^n)+(2\hat{b}+\tau(y_t))\tau(y_t^{n-1})\right.\\
		&\left.+\sum_{k=0}^{n-2}(\tau\otimes\tau^{op})\left(\left(\mathbf{1}\otimes\sqrt{\mathbf{1}+y_t^2}\right)\cdot\left(y_t^k\otimes y_t^{n-2-k}\right)\cdot\left(\mathbf{1}\otimes\sqrt{\mathbf{1}\otimes y_t^2}\right)\right)\right]dt\\
			=&n\left[\vphantom{\sum_k^n}-(1+\hat{a})\tau(y_t^n)+(2\hat{b}+\tau(y_t))\tau(y_t^{n-1})\right.\\
		&\left.+\sum_{k=0}^{n-2}\tau(y_t^k)\left(\tau(y_t^{n-k})+\tau(y_t^{n-2-k})\right)\right]dt\,.
	\end{align*}
	The first claim follows from
	$$
		\sum_{k=0}^{n-2}\tau(y_t^k)\tau(y_t^{n-k})=\sum_{k=0}^n\tau(y_t^k)\tau(y_t^{n-k})-\tau(y_t^n)-\tau(y_t)\tau(y_t^{n-1})\,.
	$$
	The second claim is immediate as compactly supported measures are determined by their moments and solutions to the ODEs \eqref{eq_free_HP_mom_recur} with matching initial conditions are unique.
\end{proof}

Lastly, we construct free Hua-Pickrell processes from more basic free stochastic processes.
This construction should be seen as the free probability analogue of Proposition 2.3 in \cite{As1}, which states that the solution of the matrix version of equation \eqref{eq_fSDE_HP} can be constructed from matrix geometric Brownian motions.
As It\^{o} formulas for matrix-valued stochastic processes and for free processes are structurally the same (see e.g. \cite{BS1,N}), it is not surprising that the referenced Proposition carries over to the free probability setting.

\begin{theorem}\label{thm_free_HP_process}
	Let $a,b\in\mathbb{R}$.
	Let $(\mathcal{A}, (\mathcal{A}_t)_{t\geq0},\tau)$ be a $W^{*}$-probability space.
	Let $(c_t=(c_{1,t},c_{2,t}))_{t\geq0}$ be a two-dimensional circular Brownian motion.
	Let $\hat{y}_0\in\mathcal{A}$ be freely independent of $(c_{t})_{t\geq0}$.
	Let $(g_t)_{t\geq0}$ be a free multiplicative Brownian motion driven by $(c_{1,t})_{t\geq0}$ with start in $\mathbf{1}$, i.e., $(g_t)_{t\geq0}$ is the unique solution to the free stochastic differential equation
	$$
		dg_t=g_t\,dc_{1,t}\,.
	$$
	Set $g_t^{(a)}:=e^{(a+1/2)t}g_t$, $t\geq0$.
	Then
	\begin{equation*}
		y_t:=\left(g_t^{(a)}\right)^{-1}\left[\hat{y}_0+\int_0^tg_s^{(a)}\,d\left(c_{2,s}+c_{2,s}^{*}+2bs\mathbf{1}\right)\,\left(g_s^{(a)}\right)^*\right]\left(g_t^{(a)}\right)^{-*}\,,\;\;t\geq0\,,
	\end{equation*}
	is a free Hua-Pickrell diffusion driven by $(\tilde{c}_t)_{t\geq0}$ with start in $\hat{y}_0$ and parameters $a,b$ where $$
		\tilde{c}_t:=-\int_0^tdc_{1,s}\,y_s(\mathbf{1}+y_s^2)^{-1/2}+\int_0^tdc_{2,s}\,(\mathbf{1}+y_s^2)^{-1/2}\,,\;\;t\geq0\,,
	$$
	is a circular Brownian motion.
\end{theorem}

\begin{proof}
	Note that as $dg_t=g_t\,dc_{1,t}$ and $dg_t^{-1}=-dc_{1,t}\,g_t^{-1}$ we have
	$$
		dg_t^{(a)}=g_t^{(a)}\,dc_{1,t}+\left(a+\frac{1}{2}\right)g_t^{(a)}dt\,,\;\;d\left(g_t^{(a)}\right)^{-1}=-dc_{1,t}\,\left(g_t^{(a)}\right)^{-1}-\left(a+\frac{1}{2}\right)\left(g_t^{(a)}\right)^{-1}dt\,.
	$$
	Using the free It\^{o} product rule (\cite{N} Theorem 3.2.5) twice, we get
	\begin{align*}
		dy_t
		=&-dc_{1,t}\,y_t+dc_{2,t}+dc_{2,t}^{*}-y_t\,dc_{1,t}^*+\left[-(2a+1)y_t+(2b+\tau(y_t))\mathbf{1}\right]dt\\
		=&d\tilde{c}_t\,\sqrt{\mathbf{1}+y_t^2}+\sqrt{\mathbf{1}+y_t^2}\,d\tilde{c}_t^*+\left[-(2a+1)y_t+(2b+\tau(y_t))\mathbf{1}\right]dt\,.
	\end{align*}
	We thus still have to show that $(\tilde{c}_t)_{t\geq0}$ is a circular Brownian motion.
	For this recall that $(\tilde{c}_t)_{t\geq0}$ is a circular Brownian motion if $(z_t=(z_{1,t},z_{2,t}))_{t\geq0}$ with $z_{1,t}:=(\tilde{c}_t+\tilde{c}_t^*)/\sqrt{2}$ and $z_{2,t}:=(\tilde{c}_t-\tilde{c}_t^*)/(i\sqrt{2})$, $t\geq0$, is a two-dimensional semi-circular Brownian motion.
	In turn, it suffices to verify the conditions of the free L\'{e}vy characterization Theorem (\cite{BCG} Theorem 6.2) for $(z_t)_{t\geq0}$.
	That each component of $z_t$ is a bounded operator and a free martingale with $z_{j,0}=\mathbf{0}$, $j\in\{1,2\}$, is clear as it is the sum of free stochastic integrals w.r.t.\ circular Brownian motions.
	Next, we show condition "3." in \cite{BCG} Theorem 6.2.
	Note that, as $(z_t)_{t\geq0}$ is a free martingale, it suffices to verify
	\begin{equation}\label{eq_pf_bougerol_id}
		\tau(A(z_{j,t}-z_{j,s})B(z_{k,t}-z_{k,s}))
		=\begin{cases}
			\tau(A)\tau(B)(t-s)\,,&j=k\,,\\
			0\,,&j\neq k\,,
		\end{cases}
	\end{equation}
	for all $0\leq s\leq t$, $A,B\in\mathcal{A}_s$ and $j,k\in\{1,2\}$.
	So let $A,B\in\mathcal{A}_s$, then
	\begin{equation}\label{eq_pf_bougerol_id2}
	\begin{split}
		&\tau\left(A(z_{1,t}-z_{1,s})B(z_{1,t}-z_{1,s})\right)\\
		=&\frac{1}{2}\tau\left[\left(\int_s^tA\left(-dc_{1,r}\,y_r+dc_{2,r}\right)(\mathbf{1}+y_r^2)^{-1/2}+\int_s^tA(\mathbf{1}+y_r^2)(-y_r\,dc_{1,r}^*+dc_{2,r}^*)\right)\right.\\
		&\cdot\left.\left(\int_s^tB\left(-dc_{1,r}\,y_r+dc_{2,r}\right)(\mathbf{1}+y_r^2)^{-1/2}+\int_s^tB(\mathbf{1}+y_r^2)(-y_r\,dc_{1,r}^*+dc_{2,r}^*)\right)\right]\\
		=&\frac{1}{2}\left(\int_s^t\left(\tau(A)\tau(y_r(\mathbf{1}+y_r^2)^{-1/2}B(\mathbf{1}+y_r^2)^{-1/2}y_r)+\tau(A)\tau((\mathbf{1}+y_r^2)^{-1/2}B(\mathbf{1}+y_r^2)^{-1/2})\right.\right.\\
		&\left.\left.+\tau(A(\mathbf{1}+y_r^2)^{-1/2}y_r^2(\mathbf{1}+y_r^2)^{-1/2})\tau(B)+\tau(A(\mathbf{1}+y_r^2)^{-1})\tau(B)\right)dr\vphantom{\int_s^t}\right)\\
		=&\tau(A)\tau(B)(t-s)\,.
	\end{split}
	\end{equation}
	The other cases in \eqref{eq_pf_bougerol_id} follow in the same way.
	Finally, we show condition "2." in \cite{BCG} Theorem 6, i.e., the bound $\tau((z_{j,t}-z_{j,s})^4)\leq K(t-s)^2$ for all $0\leq s\leq t$, $j\in\{1,2\}$, and some positive $K$.
	For this fix $s\geq0$ and set
	\begin{align*}
		\tilde{z}_{1,t}
		:=&\begin{rcases}
			\begin{dcases}
			0\,,&t\leq s\\
			z_{1,t}-z_{1,s}\,,&t>s
			\end{dcases}
		\end{rcases}
		=\int_0^t\mathbf{1}_{[s,\infty)}(r)dz_{1,r}\\
		=&-\frac{1}{\sqrt{2}}\int_0^t\mathbf{1}_{[s,\infty)}(r)\left( dc_{1,r}\,y_r(\mathbf{1}+y_r^2)^{-1/2}+(\mathbf{1}+y_r^2)^{-1/2}y_r\,dc_{1,r}^*\right)\\
		&+\frac{1}{\sqrt{2}}\int_0^t\mathbf{1}_{[s,\infty)}(r)\left(dc_{2,r}\,(\mathbf{1}+y_r^2)^{-1/2}+(\mathbf{1}+y_r^2)^{-1/2}\,dc_{2,r}^*\right)\,.
	\end{align*}
	Now, by the traced It\^{o} formula (\cite{N} Theorem 3.5.3), we have
	\begin{align*}
		&\tau(\tilde{z}_{1,t}^4)\\
		=&2\sum_{k=0}^2\int_0^t(\tau\otimes\tau^{\operatorname{op}})\left((-\mathbf{1}\otimes(\mathbf{1}+y_r^2)^{-1/2}y_r)\cdot\left(\tilde{z}_{1,r}^k\otimes\tilde{z}_{1,r}^{2-k}\right)\cdot(-\mathbf{1}\otimes y_r(\mathbf{1}+y_r^2)^{-1/2})\right)\\
		&\quad\quad\quad\;\cdot\mathbf{1}_{[s,\infty)}(r)\,dr\\
		&+2\sum_{k=0}^2\int_0^t(\tau\otimes\tau^{\operatorname{op}})\left((\mathbf{1}\otimes(\mathbf{1}+y_r^2)^{-1/2})\cdot\left(\tilde{z}_{1,r}^k\otimes\tilde{z}_{1,r}^{2-k}\right)\cdot(\mathbf{1}\otimes(\mathbf{1}+y_r^2)^{-1/2})\right)\mathbf{1}_{[s,\infty)}(r)\,dr\\
		=&2\sum_{k=0}^2\int_0^t\tau(\tilde{z}_{1,r}^k)\tau(\tilde{z}_{1,r}^{2-k})\mathbf{1}_{[s,\infty)}(r)\,dr\,.
	\end{align*}
	Note that the $k=1$ term vanishes as $\tau(\tilde{z}_{1,t})\equiv0$.
	Now let $t>s$.
	By setting $A=B=\mathbf{1}$, we know by \eqref{eq_pf_bougerol_id2} that $\tau(\tilde{z}_{1,t}^2)=t-s$.
	Plugging this in, we get
	\begin{equation*}
		\tau((z_{1,t}-z_{1,s})^4)
		=4\int_s^t(r-s)dr=2(t-s)^2\,.
	\end{equation*}
	Similarly, one also sees that $\tau((z_{2,t}-z_{2,s})^4)=2(t-s)^2$.
	This finishes the proof.
\end{proof}

\section{Appendix: Existence and uniqueness of solutions of the Hua-Pickrell SDEs}\label{appendix}

In this section we prove Theorem \ref{lem_HP_exist_unique} on the existence and uniqueness of solutions $(\tilde X_t)_{t\geq0}$ 
of the SDEs \eqref{eq_SDE_Hua-Pickrell-norm} for $\beta\in [1,\infty[$ and the ODEs  \eqref{eq_ODE_Hua-Pickrell} for $\beta=\infty$ respectively
    on $C_N^A$ with the same initial condition $\tilde X_0=\hat{x}_0\in C_N^A$.  We treat both cases
    simultaneously by interpreting the diffusion term as $0$ for $\beta=\infty$.
In principle, we can adapt the proof for $\beta=2$  in Lemma 4.2 of \cite{As2}.
	However, some arguments in the  proof in \cite{As2}  are imprecise, which are related to the application of
	Theorem 2.2 of \cite{GM}. In fact, it is required in the beginning of Section 2 in \cite{GM}
         that some function  $H$ is non-negative, where in our setting we have $H(x,y):=H_{ij}(x,y):=2(1+xy)$ for $x,y\in\mathbb{R}$, i.e., the condition
      in \cite{GM}  does not hold here.
In order to avoid this condition, we  discuss where this non-negativity is needed in \cite{GM}.
However, before we do so we need to state a minor correction to the formulas (4.3) and (4.6) in \cite{GM} as these are important in our argument.
We here state the necessary correction in terms of the notation of \cite{GM}.

Let $x\in\mathbb{R}^N$, and let
 $e_n(x):=\sum_{1\leq i_1<\dots<i_n\leq N}x_{i_1}\dots x_{i_n}$ ($n=1,\ldots,N$) be the  elementary symmetric polynomials in $N$ variables.
Moreover, let $e_0\equiv1$ and $e_{-1}\equiv0$, and
for distinct $j_1,\dots,j_k\in\{1,\dots,N\}$ we write
$$
e_n^{j_1,\dots,{j_k}}(x):=\sum_{1\leq i_1<\dots<i_n\leq N\colon i_l\neq j_m\;\forall\;l,m}x_{i_1}\dots x_{i_n}\,.
$$
  For $p\in\mathbb{N}$ consider a stochastic process $((\lambda_{1}(t),\dots,\lambda_{p}(t)))_{t\geq0}$ with values in the Weyl chamber $C_p^A$
  as well as
 the corresponding squared differences  $a_{i,j}(t):=(\lambda_{i}(t)-\lambda_{j}(t))^2$, $1\leq i<j\leq p$. Form the vector
$$
A(t):=(a_{1,2}(t),\dots,a_{1,p}(t),a_{2,3}(t),\dots,a_{2,p}(t),\dots,a_{p-1,p}(t))\in\mathbb{R}^{p(p-1)/2}\,,
$$
and set $V_{n}(t):=e_n(A(t))$ for  $t\geq0$ and  $n=1,\dots,p(p-1)/2$.\\
From now on we suppress the dependence of stochastic processes on the time $t$ in our notation.\\
In the setting of Section 4 in \cite{GM} the  process $V_n$ has the form $V_n=M_n+\int_0^tD_n(s)ds$ for some  martingale $M_n$ where 
$M_n,D_n$ can be expressed
 in terms of the $\lambda_{i}$ and some continuous functions $\sigma_{i}\colon\mathbb{R}\to\mathbb{R}$, $b_i\colon\mathbb{R}\to\mathbb{R}$, $H_{ij}\colon\mathbb{R}^2\to[0,\infty)$, $i,j\in\{1,\dots,p\}$, $j\neq i$; see equations (4.1) and (4.2) in \cite{GM}.
Moreover, as in \cite{GM}, we assume that  $H_{ij}(x,y)=H_{ji}(y,x)$ for all $x,y\in\mathbb{R}$, and we also
 use the abbreviations
\begin{equation*}
	i<j\;\;\mathrel{\widehat{=}}\;\;1\leq i<j\leq p\quad\text{and}\quad
	i\neq j\neq k\;\;\mathrel{\widehat{=}}\;\;1\leq i,j,k\leq p\colon i\neq j\neq k\neq i\,.
\end{equation*}
We now derive some correction of Eq. (4.3) in \cite{GM}.
As $e_{n-1}^{{ij}}(A)=a_{ik}e_{n-2}^{{ij},{ik}}(A)+e_{n-1}^{{ij},{ik}}(A)$, we have
\begin{align}\label{43lh}
	\sum_{i\neq j\neq k}(\lambda_i-\lambda_j)\frac{H_{ik}(\lambda_i,\lambda_k)}{\lambda_i-\lambda_k}e_{n-1}^{\bar{a}_{ij}}(A)
	=&\sum_{i\neq j\neq k}(\lambda_i-\lambda_j)(\lambda_i-\lambda_k)H_{ik}(\lambda_i,\lambda_k)e_{n-2}^{\bar{a}_{ij},\bar{a}_{ik}}(A)\notag\\
	&+\sum_{i\neq j\neq k}(\lambda_i-\lambda_j)\frac{H_{ik}(\lambda_i,\lambda_k)}{\lambda_i-\lambda_k}e_{n-1}^{\bar{a}_{ij},\bar{a}_{ik}}(A)\,.
\end{align}
The second term on the r.h.s. of \eqref{43lh} can be written as
\begin{align*}
	&\sum_{i\neq j\neq k}(\lambda_i-\lambda_j)\frac{H_{ik}(\lambda_i,\lambda_k)}{\lambda_i-\lambda_k}e_{n-1}^{{ij},{ik}}(A)\\
  =&\sum_{i<k}\sum_{j\neq i,j\neq k}\frac{H_{ik}(\lambda_{i},\lambda_k)}{\lambda_i-\lambda_k}\left((\lambda_i-\lambda_k+\lambda_k-\lambda_j)
  e_{n-1}^{{ij},ik}(A)-(\lambda_k-\lambda_j)e_{n-1}^{{kj},{ki}}(A)\right)\\
	=&\sum_{i<k}\sum_{j\neq i,j\neq k}H_{ik}(\lambda_{i},\lambda_k)\left(e_{n-1}^{{ij},ik}(A)+\frac{\lambda_k-\lambda_j}{\lambda_i-\lambda_k}\left(e_{n-1}^{{ij},{ik}}(A)-e_{n-1}^{{kj},{ki}}(A)\right)\right)\\
	=&\sum_{i<k}\sum_{j\neq i,j\neq k}H_{ik}(\lambda_i,\lambda_k)\left(e_{n-1}^{{ij},{ik}}(A)-\left(\lambda_k-\lambda_j\right)^2e_{n-2}^{{ij},{ik},{kj}}(A)-(\lambda_k-\lambda_j)(\lambda_i-\lambda_j)e_{n-2}^{{ij},{ik},{kj}}(A)\right)\\
	=&\sum_{i<k}\sum_{j\neq i,j\neq k}H_{ik}(\lambda_i,\lambda_k)\left(e_{n-1}^{{ij},{ik},{kj}}(A)-(\lambda_k-\lambda_j)(\lambda_i-\lambda_j)e_{n-2}^{{ij},{ik},{kj}}(A)\right)\,,
\end{align*}
where we used
$$
e_{n-1}^{{ij},{ik}}(A)-e_{n-1}^{{kj},{ki}}(A)
=(a_{kj}-a_{ij})e_{n-2}^{\bar{a}_{ij},\bar{a}_{ik},\bar{a}_{kj}}(A)
=(\lambda_k-\lambda_i)(\lambda_k-\lambda_j+\lambda_i-\lambda_j)e_{n-2}^{{ij},{ik},{kj}}(A)
$$
and
$$
e_{n-1}^{{ij},{ik}}
=a_{kj}e_{n-2}^{{ij},{ik},{kj}}(A)+e_{n-1}^{{ij},{ik},{kj}}(A)\,.
$$
We thus obtain the following identity which replaces (4.3) in \cite{GM}:
\begin{equation}\label{eq_appendix_GM_corr_4.3}\tag{4.3'}
	\begin{split}
		\sum_{i\neq j\neq k}(\lambda_i-\lambda_j)\frac{H_{ik}(\lambda_i,\lambda_k)}{\lambda_i-\lambda_k}e_{n-1}^{\bar{a}_{ij}}(A)
		=&\sum_{i\neq j\neq k}(\lambda_i-\lambda_j)(\lambda_i-\lambda_k)H_{ik}(\lambda_i,\lambda_k)e_{n-2}^{\bar{a}_{ij},\bar{a}_{ik}}(A)\\
		&+\sum_{i<k}\sum_{j\neq i,j\neq k}H_{ik}(\lambda_i,\lambda_k)e_{n-1}^{\bar{a}_{ij},\bar{a}_{ik},\bar{a}_{kj}}(A)\\
		&-\sum_{i<k}\sum_{j\neq i,j\neq k}(\lambda_k-\lambda_j)(\lambda_i-\lambda_j)H_{ik}(\lambda_i,\lambda_k)e_{n-2}^{\bar{a}_{ij},\bar{a}_{ik},\bar{a}_{kj}}(A)\,.
	\end{split}
\end{equation}
We next briefly discuss where (4.3) in \cite{GM} is used. First of all, in Proposition 4.2 of \cite{GM},  (4.3) must be replaced by \eqref{eq_appendix_GM_corr_4.3}.
Moreover, (4.3) in \cite{GM} is used in the the proof of Proposition 4.3 in order to derive (4.6) in \cite{GM}.
There, one considers the case that the stopping time $\tau_n=\inf\{t>0\colon V_n(t)>0\}$ is positive with positive probability.
For all $\omega\in\{\tau_n>0\}$ one then has that the finite variation part $\int_0^tD_n(\omega,s)\,ds$ of $V_n(\omega,t)$ must vanish for all $t<\tau_n(\omega)$.
Hence, one also has $D_n(\omega,t)=0$ for all $t<\tau_n(\omega)$.
By  (4.2) in \cite{GM} and by \eqref{eq_appendix_GM_corr_4.3} above, we obtain the following identity which replaces (4.6) in \cite{GM}:
\begin{equation}\label{eq_appendix_GM_corr_4.6}\tag{4.6'}
	\begin{split}
		&\sum_{i=1}^p\sigma_i^2(\lambda_i)\sum_{j\colon j\neq i}e_{n-1}^{{ij}}(A)
		+4\sum_{i<j}H_{ij}(\lambda_i,\lambda_j)e_{n-1}^{{ij}}(A)
		+2\sum_{i<k}\sum_{j\neq i,j\neq k}H_{ik}(\lambda_i,\lambda_k)e_{n-1}^{{ij},{ik},{jk}}(A)\\
		&=D_n=0\,,
	\end{split}
\end{equation}
where we used that $V_n=0$ implies that for all $j\neq i\neq k$ the terms 
$$
(\lambda_i-\lambda_j)e_{n-1}^{{ij}}(A)\,,\;\;(\lambda_i-\lambda_j)(\lambda_i-\lambda_k)e_{n-2}^{{ij},{ik}}(A)\,,\;\;
(\lambda_k-\lambda_j)(\lambda_i-\lambda_j)H_{ik}(\lambda_i,\lambda_k)e_{n-2}^{{ij},{ik},{kj}}(A)
$$
are equal to $0$.  It can be now easily checked  that the proof of Proposition 4.3 still works with \eqref{eq_appendix_GM_corr_4.6} instead of  (4.6).\\

We now return to the proof of Theorem \ref{lem_HP_exist_unique}. We  there have  to analyze where the non-negativity assumption on the $H_{i,j}$ is applied in \cite{GM};
recall that we here are interested in  $H_{ij}(x,y)=2(1+xy)$ for $i,j=1,\dots,N$ with $i\neq j$.
There are two places where this assumption is needed in \cite{GM}:
 \begin{enumerate}
\item[\rm{(1)}] In the proof of Proposition 3.5 of  \cite{GM}:
		There the implication
		$$
			\sigma_i^2(x)+\sigma_j^2(x)+H_{ij}(x,x)=0\quad\Rightarrow\quad H_{i,j}(x,x)=0
		$$
		is used where in the Hua-Pickrell setting we have $\sigma_i(x)=2\sqrt{(1+x^2)/\beta}$ for $i=1,\dots,N$.
		Clearly, the sets
                $$G_{ij}:=\{x\in\mathbb{R}\colon\sigma_i^2(x)+\sigma_j^2(x)+H_{i,j}(x,x)=0\}, \quad i,j\in\{1,\dots,N\}$$
                are empty here  which yields that  Proposition 3.5 of  \cite{GM} holds in our setting.
   \item[\rm{(2)}]             In the proof of Proposition 4.3  of  \cite{GM}:
		There,  the non-negativity assumption  enters in the context of Eq.~(4.6)  of  \cite{GM}.
		However, here Eq.~(4.6)  of \cite{GM} must be 
replaced  by \eqref{eq_appendix_GM_corr_4.6}.
Using $4\lambda_i\lambda_j=(\lambda_i+\lambda_j)^2-(\lambda_i-\lambda_j)^2$ and $(\lambda_i-\lambda_j)^2e_{n-1}^{\bar{a}_{ij}}(A)=0=(\lambda_i-\lambda_j)^2e_{n-1}^{\bar{a}_{ij},\bar{a}_{ik},\bar{a}_{jk}}(A)$ whenever $V_n=e_n(A)=0$, we now have
		\begin{align*}
			0=D_n
			=&\sum_{i=1}^p\sigma_i^2(\lambda_i)\sum_{j\colon j\neq i}e_{n-1}^{\bar{a}_{ij}}(A)
				+2\sum_{i<j}\left(4+(\lambda_i+\lambda_k)^2\right)e_{n-1}^{\bar{a}_{ij}}(A)\\
				&+\sum_{i<k}\sum_{j\neq i,j\neq k}\left(4+(\lambda_i+\lambda_k)^2\right)e_{n-1}^{\bar{a}_{ij},\bar{a}_{ik},\bar{a}_{jk}}(A)\,.
		\end{align*}
instead 	of equation (4.6) in \cite{GM}. 	Note that all terms are non-negative.
		Thus, even  when $H=H_{ij}$ is not non-negative,  the proof
                of Proposition 4.3 of \cite{GM} is still applicable.
 \end{enumerate}
 
	In summary, by the proof of Lemma 4.2 in \cite{As2}, we can apply Theorem 2.2 of \cite{GM} in our setting, 
	and  all arguments carry over to $\beta\in[1,\infty]$ with minor modifications as claimed.

\end{document}